\documentclass[draftclsnofoot]{IEEEtran}
\onecolumn

\usepackage{graphicx}
\usepackage{amsmath}
\usepackage{hyperref}
\usepackage{amssymb}
\usepackage{amsthm}
\usepackage{amsopn}
\usepackage{mathrsfs}
\usepackage{mathtools}
\usepackage{stmaryrd}
\usepackage[T1]{fontenc}
\usepackage{lipsum}
\usepackage{graphicx}
\usepackage{enumerate}
\graphicspath{ {/home/user/images/} } 
\usepackage{subcaption}
\usepackage{cite}
  
\usepackage{tikz}
  \usetikzlibrary{matrix}
  \usetikzlibrary{fit}
  \usetikzlibrary{backgrounds}
  \usetikzlibrary{arrows}
  \usetikzlibrary{shapes}
  \usetikzlibrary{decorations.pathmorphing}
\usepackage{tikz-cd}

\DeclareMathOperator{\rank}{rank}

\DeclareMathOperator{\End}{End}

\DeclareMathOperator{\GL}{GL}

\DeclareMathOperator{\id}{id}
\DeclareMathOperator{\vol}{vol}

\DeclareMathOperator{\Pfaff}{Pfaff}
\DeclareMathOperator{\Isom}{Isom}
\DeclareMathOperator{\Lie}{Lie}

\DeclareMathOperator{\argmin}{argmin}
\DeclareMathOperator{\Ad}{Ad}

\DeclareMathOperator{\U}{U}
\DeclareMathOperator{\Sp}{Sp}
\DeclareMathOperator{\Li}{Li}

\newtheorem{defi}{Definition}
\newtheorem{ex}{Example}
\newtheorem{rem}{Remark}
\newtheorem{lem}{Lemma}
\newtheorem{prop}{Proposition}

\begin{document}


%

%

%
\ifCLASSINFOpdf
\else
\fi
\hyphenation{op-tical net-works semi-conduc-tor}

%
\title{Gaussian distributions on Riemannian symmetric spaces, random matrices, and planar Feynman diagrams}
%
%
%

\author{Simon~Heuveline,
        Salem Said, 
        and Cyrus Mostajeran
\thanks{S. Heuveline is with the Centre for Mathematical Sciences of the University of Cambridge}
\thanks{S. Said is with the University of Bordeaux}
\thanks{C. Mostajeran is with the Department of Engineering of the University of Cambridge}}
\maketitle

\begin{abstract}
 Gaussian distributions can be generalized from Euclidean space to a wide class of Riemannian manifolds. Gaussian distributions on manifolds are harder to make use of in applications since the normalisation factors, which we will refer to as partition functions, are complicated, intractable integrals in general that depend in a highly non-linear way on the mean of the given distribution. Nonetheless, on Riemannian symmetric spaces, the partition functions are independent of the mean and reduce to integrals over finite dimensional vector spaces. These are generally still hard to compute numerically when the dimension (more precisely the rank $N$) of the underlying symmetric space gets large. 
On the space of positive definite Hermitian matrices, it is possible to compute these integrals exactly using methods from random matrix theory and the so-called Stieltjes-Wigert polynomials. In other cases of interest to applications, such as the space of symmetric positive definite (SPD) matrices or the Siegel domain (related to block-Toeplitz covariance matrices), these methods seem not to work quite as well. Nonetheless, it remains possible to compute leading order terms in a large $N$ limit, which provide increasingly accurate approximations as $N$ grows. This limit is inspired by realizing a given partition function as the partition function of a zero-dimensional quantum field theory or even Chern-Simons theory. From this point of view the large $N$ limit arises naturally and saddle-point methods, Feynman diagrams, and certain universalities that relate different spaces emerge.
\end{abstract}


%
\IEEEpeerreviewmaketitle

\section{Introduction}

We begin by briefly motivating the use of symmetric spaces in applications within information theory through two main examples used in this paper: the space of symmetric positive definite (SPD) matrices and the Siegel domain. This will then lead to a presentation of our theory of Gaussian distributions on symmetric spaces and the main results of our paper.

\subsection{Motivation}
\label{subsec:Moti} 

The Riemannian symmetric space of most interest to many current applications is the space of symmetric positive definite (SPD) $N\times N$ matrices $\mathcal{P}_N$, which admits the quotient representation $\GL(N, \mathbb{R})/\textrm{O}(N)$ from the viewpoint of Lie theory. Due to the enormous range of applications that generate data within the SPD manifold $\mathcal{P}_N$, there has been intense and growing research interest in the development of techniques for the analysis, representation, and classification of such data with applications in medical imaging \cite{BASSER1994,Bihan2001,Batchelor2005,Pennec2006}, shape analysis \cite{Fillard2007}, radar signal processing \cite{Arnaudon2011,Arnaudon2013},  computer vision \cite{Tuzel2006,Tuzel2008,Pang2008,Caseiro2011,Jay2013,Guo2013,Harandi2016}, and continuum mechanics \cite{Moakher2006}. In more recent years, the interest in probabilistic models on manifolds driven by applications in geometric learning has led to the development of theories for the construction of probability distributions such as Gaussians on manifolds including $\mathcal{P}_N$ \cite{PennecGauss2006,Cheng2013,Sa16,Sa17,HSM21}.

A sample application can be found in the context of electroencephalogram (EEG) based brain-computer interfaces (BCI) \cite{Ba2012,Zanini2018}, where data measurements taken from $N$ electrodes at $T$ time points generate a matrix $X\in \mathbb{R}^{N \times T}$. To average over time, one views the empirical covariance matrix $\frac{1}{T-1}X^TX\in \mathcal{P}(N)$ as signal descriptor. After assuming that many points in $\mathcal{P}(N)$ arising in this way (i.e. many different sessions with the same subject) follow a Gaussian distribution on $\mathcal{P}(N)$ with mean $\bar{x}$ and variance $\sigma^2$, parameter estimation of $\bar{x}$ and $\sigma^2$ becomes a natural question.
From Definition \ref{defi:partition}, it will become clear that for given data points $\{x_i\in \mathcal{P}(N):i\in \{1, \dots, n\}\}$, the log-likelihood function of a Gaussian distribution looks like \begin{align}
\label{eq:mle}
    L(\bar{x}, \sigma)=-n \log(Z_{\mathcal{P}(N)}(\sigma))- \frac{1}{2\sigma^2} \sum \limits_{k=1}^n d^2_g(\bar{x}, x_k),
\end{align}
where $d_{g}$ is a Riemannian distance function on $\mathcal{P}(N)$ and $Z_{\mathcal{P}(N)}(\sigma)$ is the normalisation constant of the Gaussian distribution and the subject of much of this work.
Hence, finding the maximum likelihoood estimate (MLE) of $\bar{x}$ can be done by computing the \textbf{Riemannian barycentre}, $ \underset{x \in M}{\argmin}\sum \limits_{k=1}^n d^2_g(x, x_k)$, a task for which there exist an increasing number of high-performance routines.
The harder problem of finding the MLE of $\sigma$ on $\mathcal{P}(N)$ is then solved by computing $Z_1(\sigma)\coloneqq Z_{\mathcal{P}(N)}(\sigma)$, the \textbf{partition function} (up to a $\sigma$-independent multiplicative constant) which then enables us to compute the saddle point of equation \eqref{eq:mle}.

Another example concerns the space $\mathcal{T}_n^N$ of $Nn\times Nn$ Block-Toeplitz covariance matrices consisting of $n^2$ many $N\times N$ Toeplitz blocks (as defined in \cite{Sa17}, IV. C.). $\mathcal{T}_n^N$ plays an important role in multi-channel and two-dimensional linear prediction and filtering problems \cite{Therrien1981,Jakobsson2000}. This space is reducible \begin{align}
    \mathcal{T}_n^N\cong\mathcal{T}_n^1 \times \underbrace{\mathbb{D}(N) \times\dots \times \mathbb{D}(N)}_{n \textrm{ times}}
\end{align}
with factors given by the \textbf{Siegel domain} $\mathbb{D}(N) \cong \Sp(2N, \mathbb{R})/\U(N)$.
As partition functions factorize nicely for such products (see \cite{Sa17}, (20)-(22)) and $Z_{\mathcal{T}_1^N}$ has been computed exactly in \cite{Sa17} (section IV. B.), the only
remaining thing to do is to compute $Z_S \coloneqq Z_{\mathbb{D}(N)}$. $\mathbb{D}(N)$ itself has interesting applications and  provides a geometric setting which leads to deep results with regard to linear filtering and linear control theory \cite{Bougerol1993}.

In \cite{Sa16} (resp. \cite{Sa17}) $Z_1(\sigma)$ (resp. $Z_S(\sigma)$) have been computed numerically for $N \approx 20$ using specifically designed Monte Carlo methods. Such numerical methods remain effective only up to around $N\approx 40$, even though cases involving $N > 250$ are of relevance to applications such as EEG based brain-computer interfaces \cite{Ba2012,Se17}. In general, any application involving covariance descriptors generated by measurements from a large number of sensors will present similar computational challenges. This motivates us to search for leading order terms in the so-called \textbf{large $N$ limit} where $N\rightarrow \infty$ while $t=N\sigma^2$ is fixed.

\subsection{Outline}

In this work, we will not be able to give closed exact formulas for the desired partition functions $Z_1$ and $Z_S$ (without computing certain skew-orthogonal polynomials numerically) but we will try to get as close as possible. In terms of direct usefulness to the applications of Section \ref{subsec:Moti} (which will be elaborated on in future work) our most important contributions are the following.
\begin{enumerate}
\item
Results on how to compute $Z_1$ and $Z_S$ numerically for finite $N$ which can be quite large (but not too large). Here, we still have to compute $N$ so-called skew-orthogonal polynomials numerically.
\item Two inequivalent explicit formulas for the large $N$ limit of $Z_1$ with different interpretations.
\item An integral equation whose solution completely describes the large $N$ limit of $Z_S$ and partition functions of other spaces.
\end{enumerate}
On our way to these results, we will learn much more about several other symmetric spaces and we will develop many different perspectives on the partition functions.

In Section \ref{sec:generalGauss}, we set up the general framework of Gaussian distributions on Riemannian symmetric spaces and note the significance of symmetric spaces as manifolds on which the partition functions become tractable. We will also introduce relevant examples, which are mostly non-compact spaces.
In Section \ref{sec:finite_N}, using techniques from random matrix theory,  we will give an exact formula for $Z_2(\sigma)$, the partition function of the space of positive definite Hermitian $N\times N$ matrices $\mathcal{P}_{\mathbb{C}}(N)=\GL(N,\mathbb{C})/U(N)$. The existence of such a formula is very specific to the case $\mathcal{P}_{\mathbb{C}}(N)$, as certain relevant orthogonal polynomials, the Stieltjes-Wiegert polynomials, are well known explicitly in this case. Such formulas for the partition functions of interest, $Z_1$ and $Z_S$, do exist in theory, but still depend on coefficients of so-called skew-orthogonal polynomials, which are not known explicitly to us. This will be discussed anyway, as the skew-orthogonal polynomials may be computed numerically (for $N$ not too large), by using a symplectic Gram-Schmidt algorithm. The computation of $Z_2$ using random matrix theory has been previously carried out in the physics literature \cite{Ma05,Ti04} and the relevance of skew-orthogonal polynomials relating to $\GL(N, \mathbb{R})$ has been observed independently in \cite{Ti20}. Many further results were also found independently in the wonderful paper \cite{Ti20} from a slightly different perspective, on which we will try to comment throughout this work. In any case, we will try to be as complementary to \cite{Ti20} as possible. Another recent beautiful paper on the topic is \cite{Fo20}, which considers the case $\mathcal{P}_{\mathbb{C}}(N)$ in much greater detail.

Section \ref{sec:large_N} is at the heart of this paper and discusses the limit in which any known numerical methods will start to break down: The large $N$ limit $N \rightarrow \infty$ (while keeping $t=N\sigma^2$ fixed). First, we will use the results from section \ref{sec:finite_N} to compute the large $N$ limit of $Z_2$ from elementary methods.
Then, we will develop a more powerful description of the large $N$ limit: The saddle-point equation. This part makes use of many insights from the physics literature and we will make a physical interpretation in terms of planar Feynman diagrams quite explicit. In addition to qualitative interpretations, we will derive an integral equation out of this, whose solution, the master field, fully characterizes the large $N$ limit. This equation for $Z_1$, $Z_2$ (and $Z_4$, the partition function of the space of quaternionic hermitian matrices $\GL(N,\mathbb{H})/\Sp(N)$) is related by a scaling, so that their respective large $N$ limit is essentially the same, a phenomenon we will refer to as \textbf{universality}.
This allows us to relate the large $N$ limit of $Z_2$ from elementary calculation to that of $Z_1$. Moreover, the solution to the saddle-point equation is given explicitly, which gives another formula for the large $N$ limit. In the case of $Z_S$, we do not know of an explicit solution to the saddle-point equation, but we will write down the equation anyway, which might be solved either numerically or analytically in future work. Moreover, we observe another universality relating the space $\textrm{SO}(2N)/\textrm{O}(N)$ (resp. its non-compact dual) listed as DIII, that we learned about from \cite{Ti20} (Appendix C) to the Siegel domain.

In Section \ref{sec:duality}, we discuss a further direction of enquiry, based on the idea that partition functions of dual pairs of non-compact and compact spaces should be related. This is rigorously demonstrated for $Z_2$, with the dual pair $M = \mathcal{P}_{\mathbb{C}}(N)$ and $M^\vee = \mathrm{U}(N)$. The conclusion in Section \ref{sec:Outlook} provides an outlook for future research directions.

\section{Gaussian Distributions on Riemannian manifolds and symmetric spaces}
\label{sec:generalGauss}

In this section we will define the objects of interest, \textbf{partition functions}, which in this context are normalisation factors of Gaussian distributions on manifolds. On a Riemannian symmetric space $M=G/H$, the partition function simplifies to an integral over a $\rank(G/H)$-dimensional vector space, using generalized polar coordinates. We will briefly review enough basics about the structure theory of Riemannian symmetric spaces to understand this step. As we will focus on specific examples in later sections, we will not go into many important details of the general theory.
A more detailed review, which is enough for our purposes, can be found in \cite{Sa16,Sa17} and much more on the general theory can be found in the standard reference \cite{Hel79}.

\subsection{Gaussian distributions on general Riemannian manifolds and symmetric spaces}
Any Riemannian manifold $M$ with metric tensor $g$ is naturally equipped with a Riemannian distance function $d_g: M \times M \rightarrow \mathbb{R}_{\geq 0}$, which is a metric function on $M$ in the sense of metric spaces. This allows us to define a Gaussian distribution on $(M,g)$ in the following way:

\begin{defi}
\label{defi:partition}
Let $(M,g)$ be a Riemannian manifold on which \begin{equation}
\label{eq:partfct}
    Z(\bar{x}, \sigma)= \int_M \exp \big(-\frac{d_g^2(x,\bar{x})}{2\sigma^2} \big) d \vol_g(x)< \infty
\end{equation}
for any $\bar{x} \in M$, $\sigma \in \mathbb{R}_{>0}$. We define the \textbf{Gaussian distribution} on $M$ with \textbf{mean} $\bar{x}$ and \textbf{standard deviation} $\sigma$ to be given by the density function
\begin{equation} \label{eq:gaussdensity}
    p(x|\bar{x}, \sigma)=\frac{1}{Z(\bar{x},\sigma)} \exp \big(-\frac{d_g^2(x,\bar{x})}{2\sigma^2} \big)
\end{equation}
with respect to the Riemannian volume measure $d \vol_g$.
We call $Z$ the \textbf{partition function} of the distribution.
\end{defi}

\begin{rem}
Gaussian distributions on Euclidean spaces have a number of defining properties, each of which leads to the same distribution. Examples of these characteristic properties include the observation from kinetic theory by Einstein that the position of a particle undergoing Brownian motion follows a Gaussian distribution, or the celebrated result from information theory by Shannon that among all probability distributions with a given mean and variance, the distribution with maximum entropy is Gaussian. Interestingly, there is no guarantee that these different defining properties will lead to the same definition of a Gaussian distribution in non-Euclidean spaces. For instance, the heat kernel of a Riemannian manifold may not lead to a definition of a Gaussian with the familiar and desirable statistical properties found in Euclidean spaces. The most attractive feature of Definition \ref{defi:partition} for a Riemannian Gaussian is that it ensures that maximum likelihood estimation is equivalent to the Riemannian barycentre problem \cite{Sa16,Sa17}, which is a fundamental and important statistical property of Gaussians on Euclidean domains. 
\end{rem}

On a general Riemannian manifold, $Z$ will depend in a highly non-linear and intractable way on the mean $\bar{x}$ as the dominant contribution to Equation (\ref{eq:partfct}) will come from local data such as the curvature at $\bar{x}$. This complex dependence on local data is also observed in the heat kernel approach to the Atiyah-Singer Index theorem (see chapters 2 and 4 of \cite{BGV13}), for instance. Hence, to have a chance of actually computing the full partition function analytically, we should restrict to spaces that look the same locally at any point. Riemannian symmetric spaces formalize this intuition and we are fortunate that precisely such spaces appear in most applications of interest (see \cite{Sa16,Sa17}).

\begin{defi}
\label{defi:RSS}
A \textbf{Riemannian symmetric space} is a Riemannian manifold $(M,g)$ such that for each point $p \in M$ there is a global isometry $s_p:M \rightarrow M$ such that $s_p(p)=p$ and $d _p s_p=-\id_{T_pM}$, where $d _p s_p$ denotes the differential of $s_p$ at $p$.
\end{defi}
This definition guarantees the existence of many isometries on $M$, ensuring that $M$ is highly symmetric. In the theory of symmetric spaces the statement that "every point looks the same locally" can be formalized by a theorem that in any (locally) symmetric space the curvature tensor is parallel~\cite{Hel79}. Another characterization that will be crucial for us is the following Lie-theoretic perspective:

\begin{prop}
\label{thm:Lie_perspective}
For a symmetric space $M$, the connected component $G$ of the Isometry group of $M$ has a Lie group structure, such that it acts transitvely on $M$. Let $H$ be the stabilizer group of a point $p\in M$ under the action of $G$, then $H$ is a compact subgroup of $G$, and $M\cong G/H$.
\end{prop}

In this proposition in particular, we have dropped many important details that will not play a role for the concrete examples in later sections (see \cite{Eb97,Hel79}).

\begin{rem} \label{rem:z_invariance}
In the language of Proposition \ref{thm:Lie_perspective}, it can be shown that the partition function on a symmetric space does not depend on $\bar{x}$ (see \cite{Sa17}, Proposition 1), which was our original motivation. In particular, since $G$ acts transitively on $M$, let $\bar{x}_1=g\bar{x}_2$ where $\bar{x}_1,\bar{x}_2 \in M$, $g\in G$ and observe that
\begin{align*}
Z(\bar{x}_1, \sigma)= \int_M p(x|\bar{x}_1) d \vol_g=\int_M p(x|g\bar{x}_2) d \vol_g=\int_M p(g^{-1}x|\bar{x}_2) d \vol_g=\int_M p(x|\bar{x}_2) d \vol_g=Z(\bar{x}_2, \sigma),
\end{align*}
where we have used the invariance of the measure $d\vol_g$ under $G$.
\end{rem}

The perspective of Proposition \ref{thm:Lie_perspective} is usually introduced to classify symmetric spaces by making use of the classification of (semi-)simple Lie groups. We will briefly touch on such a classification and mention which classes of symmetric spaces will be considered here. A more systematic approach starting from classification of symmetric spaces is due in future work.

For our purposes, the perspective of Proposition \ref{thm:Lie_perspective} is also important as it drastically simplifies the partition functions on $M=G/H$ through the following proposition. See Section II and Appendix A of \cite{Sa17} as well as references therein for a discussion. For any symmetric space $M$, the Lie algebra $\mathfrak{g}$ of its isometry group $G=\Isom_0(M)$ has a natural involution (induced from the global symmetries $s_p$ of Definition \ref{defi:RSS}), which introduces a \textbf{Cartan decomposition} $\mathfrak{g}= \mathfrak{h} \oplus \mathfrak{p}$ \cite{Hel79}. In the following proposition, we will integrate over a maximal abelian subspace of $\mathfrak{p}$, which we will call a \textbf{Cartan subspace}.

\begin{prop}
\label{thm:int_sym}
Given an integrable function on a non-compact symmetric space $(M=G/H,g)$, we can integrate in the following way: 
\begin{align}
    \int_{M}f(x)d\vol_g(x)= C \int_H \int_{\mathfrak{a}}f(a,h)D(a)da dh
\end{align}
where $dh$ is the normalised $H$-invariant measure on $H/H_{\mathfrak{a}}$, with $H_{\mathfrak{a}}$ the centralizer of $\mathfrak{a}$ in $H$, and $da$ the Lebesgue measure on a Cartan subspace $\mathfrak{a}\subset \mathfrak{g}=\Lie(G)$. The function $D:\mathfrak{a}\rightarrow\mathbb{R}^+$ is given by the following product over roots $\lambda: \mathfrak{a}\rightarrow\mathbb{R}$ with respective multiplicities $m_{\lambda}$, $D(a)=\prod_{\lambda>0} \sinh^{m_\lambda}(|\lambda(a)|)$.
\end{prop}

Even though, it has more structure, we will view $\mathfrak{a}$ just as a real vector space, whose dimension defines the \textbf{rank of $G/H$}.
This fact lets us simplify the partition function on a symmetric space to the form \begin{align}
\label{eq:Z_general}
    Z(\sigma)=Z(\bar{x}, \sigma)= \frac{\omega(S)}{|W|}\int_{\mathfrak{a}} \exp\big(-\frac{B(a,a)}{2 \sigma^2}\big) \prod_{\lambda \in \Delta^+}|\sinh{\lambda(a)}|^{m_{\lambda}} da
\end{align}
where $B$ is the Killing form on $\mathfrak{a} \subset \Lie (G)$ and $\Delta^+$ is the set of positive roots $\lambda \in \mathfrak{a}^*$. The appearance of the Killing form of $\Lie (G)$ comes from the fact that it defines an $\Ad(G)$-invariant non-degenerate symmetric bilinear form on 
$\mathfrak{g} = \mathrm{Lie}(G)$ which descends to the metric on $M$. The constant $\omega(S)$ will be irrelevant to us as it does not depend on $\sigma$. 
Equation \eqref{eq:Z_general} is the starting point for the computations in following sections, where it will become much more concrete.

\subsection{The spaces of interest}

Using results such as Proposition \ref{thm:Lie_perspective} and the structure theory of Lie groups, it can be shown that there is essentially a finite list of irreducible symmetric spaces. They fall in three classes: spaces of non-compact type (non-positive curvature), of Euclidean type (vanishing curvature), and of compact type (non-negative curvature). The spaces that we will introduce in this paper, are of non-compact type (up to Euclidean factors). Note that we certainly do not exploit the whole list of symmetric spaces. Rather, we are mostly considering the ones of interest to the applications from Section \ref{subsec:Moti} and related spaces. In \cite{Ti20} Appendix B similar and further spaces have been considered and embedded into the classification scheme of \cite{Zi11}. 

There is a duality, which relates a symmetric spaces of non-compact type $M$, to a dual symmetric space of compact type $M^\vee$. In examples, it turns out that the partition functions of mutually dual symmetric spaces are related. This will be introduced in Section \ref{sec:duality}.

\begin{defi}
\label{def:Symm_exp}

\begin{enumerate}
    \item
    \label{def:SPD}
    The spaces of symmetric, Hermitian and quaternionic Hermitian positive definite matrices will be denoted by $\mathcal{P}_{\mathbb{F}}(N)\cong \GL(N, \mathbb{F})/K$ where $K\in \{\U(N),\textrm{O}(N),\Sp(N)\}$ for $\mathbb{F}\in \{\mathbb{R}, \mathbb{C}, \mathbb{H}\}$, respectively. In each case, $\mathfrak{a} \subset \Lie(\GL(N, \mathbb{F}))\cong \End(N, \mathbb{F})$ is the $N$-dimensional space of real diagonal matrices. We will denote the partition functions by $Z_{\beta}(\sigma)$ where $\beta\in \{1,2,4\}$, respectively. $\mathcal{P}(N)\coloneqq \mathcal{P}_{\mathbb{R}}(N)$ is commonly referred to as the \textbf{space of SPD matrices}.
    \item
    \label{def:Siegel}
    The \textbf{Siegel domain} $\mathbb{D}(N)\cong \Sp(2N, \mathbb{R})/\U(N)$ is the space of complex symmetric $N\times N$ matrices $z$ such that the imaginary part $\textrm{Im}(z)$ is a positive definite matrix. We will denote its partition function by $Z_S$.
\end{enumerate}
\end{defi}

\begin{rem}
\begin{enumerate}

    \item
For the spaces $\mathcal{P}_{\mathbb{K}}(N)$, the parameter $m_{\lambda}$ from Equation \eqref{eq:Z_general} takes the form $m_\lambda=:\beta \in \{1,2,4\}$ corresponding to the real dimensions $1,2$ and $4$ of the spaces $\mathbb{R}$, $\mathbb{C}$ and $\mathbb{H}$. Calling this parameter $\beta$ is common in statistical physics and the random matrix literature \cite{Me04} and we will continue with this notation. The restriction of $\beta$ to these three values famously goes by the name of Dyson's threefold way (see \cite{Dy623}). All of these spaces belong to class $A$ in classifications of symmetric spaces (see \cite{Zi11}).
\item The Siegel domain goes beyond Dyson's initial classification and is referred to as a post-Dyson class in \cite{Zi11}. We will see that in the sense of random matrix theory $\beta=1$ will hold for the Siegel domain as well. It is identified to be $CI$ in the classification scheme of \cite{Zi11}.

\item
In some cases, such as hyperbolic $N$-space $\mathbb{H}(N)=\textrm{SO}(N,1)/\textrm{SO}(N)$, the rank does not depend on $N$:
\begin{align}
\rank(\mathbb{H}(N))=1 
\end{align}
for all $N$ so that $Z$ can be computed just by integrating over $\mathbb{R}$:
\begin{align}
    Z_{\mathbb{H}(N)}(\sigma)=\frac{\omega_{N-1}}{2} \int_{-\infty}^{\infty} \exp \left(-\frac{r^2}{2\sigma^2} \right)|\sinh(r)|^{N-1}dr.
\end{align}
This has been solved analytically in \cite{Sa17}.
However, for the spaces of interest from Definition \ref{def:Symm_exp}, we have $rank(M)\rightarrow\infty$ while $\dim(M) \rightarrow \infty$, which makes the integral $Z$ from Equation \eqref{eq:Z_general} difficult to compute numerically in high dimensions. In general, there is a dichotomy between spaces of rank $1$ and higher rank, which is well studied in the differential geometry literature \cite{Eb97}.
\end{enumerate}
\end{rem}

\subsubsection{The threefold way - Class A}
In the $A$-class, the positive root space $\Delta^+$ consists of $\lambda(a)=a_{ii}-a_{jj}$ for $i<j$, so that \eqref{eq:Z_general} gives us (see \cite{Sa17}, IV, A. and Appendix A, Example 1 for the case $\beta=2$) \begin{align}
    Z_{\beta}(\sigma)= \frac{\omega_{\beta}(N)}{N!}\int_{\mathfrak{a}}\prod \limits_{i=1}^N \exp \Big(-\frac{2a_{ii}^2}{\sigma^2}\Big) \prod_{i<j}|\sinh(a_{ii}-a_{jj})|^{\beta} da
\end{align}
where $\omega_{\beta}(N)=\omega(S)$ and $da$ is the Lebesgue measure on $\mathfrak{a}\cong \mathbb{R}^N$. Here, $\omega(S)$ is the same as in Equation (\ref{eq:Z_general}).
After a change of variables $u_i=\exp(2a_{ii})$, we end up with \begin{align}
\label{al:Z_beta}
    Z_{\beta}(\sigma)=\frac{C_{N,\beta}(\sigma)}{(2\pi)^NN!}  \int_{\mathbb{R}^N_+}\prod_{i=1}^N \Big(\exp\big(-\frac{\log^2(u_i)}{2\sigma^2}\big)\Big) |\Delta(u)|^{\beta} \prod_{i=1}^N du_i
\end{align}
where $C_{N,\beta}(\sigma)\coloneqq \frac{\omega_{\beta}(N) (2\pi)^N}{2^{NN_\beta}}\exp\Big(-NN_{\beta}^2 \frac{\sigma^2}{2}\Big)$ and
$\Delta$, the \textbf{Vandermonde determinant} given by $\Delta(u)\coloneqq \prod_{i<j} (u_j-u_i)$ and $N_{\beta} \coloneqq \frac{\beta}{2}(N-1)+1$. We included the $(2\pi)^N$ into $C_{N,\beta}$ for later convenience which will become clear in Section \ref{sec:large_N}.

\begin{rem}
The integral in \eqref{al:Z_beta} takes the standard form of a random matrix ensemble \cite{Me04} with potential given by $V_{SW}(x;\sigma)=\frac{1}{2 \sigma^2}\log^2(x)$. If instead we consider the quadratic potential $V_Q(x;\sigma)=\frac{1}{2\sigma^2}x^2$, the integral would simply correspond to the well-known orthogonal, unitary and symplectic ensembles (for $\beta \in \{1,2,4\}$ respectively) studied in classical random matrix theory \cite{Me04}. The analogy to these ensembles will be further discussed in Remark \ref{rem:semicirc}.
\end{rem}
\begin{rem}
In fact, matrix models with the potential $V_{SW}$ also appear in the physics literature as the partition functions of $\U(N)$ Chern-Simons theory on $S^3$ (\cite{Ma05,Ti04}). 
This relation is not directly of use here, but it is worth mentioning that it provided our original inspiration for probing the structure behind the normalising factors.
Moreover, the large $N$ limit of $\U(N)$ Chern-Simons theory is of physical interest and has been well studied in the theoretical and mathematical physics literature such as \cite{Ho03} Chapter 36.2 and \cite{AK07}, which motivates Section \ref{sec:large_N}. This and many further physical interpretations have recently also been discussed in \cite{Fo20}.
\end{rem}


\subsubsection{The Siegel disc - Class C}
Now, to compute $Z_S$, we consider the root system of the Lie algebra of the symplectic group $\Sp(2N, \mathbb{R})$, which looks quite different: Next to roots of the form $a_{ii}-a_{jj}$ for $i<j$ there are also roots of the form $a_{ii}+a_{jj}$ and $2a_{ii}$ and next to permuting the $a_{ii}$, we have the additional symmetry of the root system given by $a_{ii}\mapsto -a_{ii}$ so that the Weyl-group of $\mathfrak{sp}(N)$ is seen to be of order $|W|=2^NN!$. Moreover, $\omega(S)=\frac{\vol(\U(N))}{2^N}$ which fixes the prefactor.
As discussed in \cite{Sa17}, IV., C., this yields: 
\begin{align}
 Z_S(\sigma) =\frac{\vol (\U(N))}{2^{2N}N!}\int_{\mathbb{R}^N}
\prod \limits_{i=1}^N \exp \big(- \frac{a_{ii}^2}{2\sigma^2}\big) \prod \limits_{i<j} \sinh |a_{ii}-a_{jj}|\prod \limits_{i\leq j} \sinh |a_{ii}+a_{jj}| \prod \limits_{i=1}^N da_{ii}.
\end{align}
Now, introducing the parameter $u_i=\frac{1}{2}(\exp(2a_{ii})+ \exp(-2a_{ii}))= \cosh(2a_{ii})$, we find after a short calculation that $Z_S$ takes the form
\begin{align}
\label{eq:Z_S}
  Z_S(\sigma)=\frac{\vol (\U(N))2^{\frac{N(N+1)}{2}}}{N!}\int_{(1, \infty)^{N}}
\left( \prod \limits_{i=1}^N \exp \big(- V_{S}(u_i;\sigma)\big)\right)| \Delta(u)| \prod \limits_{i=1}^N du_i
\end{align}
where $V_S$ is the potential of the Siegel domain defined on $u \in (1, \infty)$ by
\begin{align}
    \label{eq:Pot_Siegel}
    V_S(u;\sigma)=\frac{\log^2(u+\sqrt{u^2-1})}{8\sigma^2}.
\end{align}
Note that $Z_S$ has the same form as $Z_1$ except with an even more complicated potential. Partition functions similar to \eqref{eq:Z_S} with $\beta=2$ as exponent of $\Delta$ have been commented on in \cite{Ti20} (page 23) and might be of interest in obtaining information about $Z_S$.

Let us now describe methods from random matrix theory on how to solve these integrals analytically at finite $N$ values. 

\section{Exact formulas at finite $N$ and (skew-)orthogonal Polynomials}
\label{sec:finite_N}

In this section we will use orthogonal polynomials (specifically, Stieltjes-Wigert polynomials) to solve the easiest case, $Z_2$, which is less challenging but also less interesting for applications than $Z_1$ (SPD matrices). This will be useful for applications anyway, as $Z_1$ and $Z_2$ are related in the large $N$ limit. Moreover, we will discuss a method for computing $Z_1$ and $Z_S$ in terms of coefficients of so-called skew-orthogonal polynomials, which can be found numerically by a symplectic Gram-Schmidt algorithm. The techniques are known from the classical ensembles in random matrix theory \cite{Me04}.

\begin{defi}

\label{defi:ortho_pol}
Let $V:(a,b)\rightarrow \mathbb{R}$ be a given potential. A set of polynomials $\{R_i:i=1, \dots ,N\}$, with $R_j(x)=a_jx^j+ \dots$ of degree $j$ is called 

\begin{enumerate}
    \item 
\textbf{orthogonal with potential $V$} if they form an orthonormal basis of the space of degree $N$ polynomials with respect to the inner product \begin{align}
   \langle f,g\rangle_2=\int_a^b \exp\big(-V(x)\big) f(x)g(x)dx
\end{align}

\item
\textbf{skew-orthogonal with potential $V$} if they bring the skew-symmetric inner product  
\begin{align}
   \frac{1}{2} \langle f,g\rangle_1=\int_a^b \int_a^b f(x)g(y)\operatorname{sign}(x-y) \exp (-V(x)) \exp(-V(y))dxdy
\end{align}
to the standard form, meaning \begin{align}
    \langle R_{2k},R_{2l} \rangle_1&=\langle R_{2k+1},R_{2l+1} \rangle_1=0\\
    \langle R_{2k},R_{2l+1} \rangle_1&=-\langle R_{2l+1},R_{2k} \rangle_1=\delta_{kl}.
\end{align}
\end{enumerate}
\end{defi}

\subsection{An exact formula for $Z_2$} \label{subsec:exactz2}

In the classical unitary ensemble of \cite{Me04} with potential $V_{Q}(x; \sigma)=\frac{1}{2\sigma^2}x^2$ on $\mathbb{R}$, it is well known that the (rescaled) Hermite polynomials are orthogonal with respect to $V_Q$. For our potential $V_{SW}(x;\sigma)=\frac{1}{2\sigma^2}\log^2(x)$ on $\mathbb{R}_{>0}$ the respective orthogonal polynomials are less standard but also well known (see \cite{Sz39}, 2.7):
\begin{lem}
\label{Lem:SW}
To match notation with \cite{Sz39}, denote $q\coloneqq e^{\sigma^2}$. The so called \textbf{Stieltjes-Wieger polynomials} given by  \begin{align}
P^{SW}_n(x;\sigma)\coloneqq \frac{(-1)^n q^{\frac{n}{2}+\frac{1}{4}}}{\Big(\prod \limits_{i=1}^n (1-q^i)\Big)^{\frac{1}{2}}} \sum \limits_{j=0}^n \left[\begin{array}{cc}
     n \\
     j
\end{array}\right]_q q^{j^2} (-q^{\frac{1}{2}}x)^j
\end{align}
are orthogonal with potential $V_{SW}(x;\sigma)=\frac{1}{2\sigma^2}\log^2(x)$. Here, we have used the notation  \begin{align}
   \left [\begin{array}{cc}
     n \\
     j
\end{array}\right]_q \coloneqq \frac{[n]_q!}{[j]_q![n-j]_q!}
\end{align}
where \begin{align}
    [n]_q!\coloneqq \prod \limits_{k=1}^n (1-q^k)
\end{align}
\end{lem}

It is a well known trick \cite{Me04} that we can rewrite the Vandermonde determinant as \begin{align} \label{eq:vmondorthpo}
    \Delta(u)=\det (u_i^{j-1})= \frac{1}{\prod _{k=1}^N a_k}\det(P_{j-1}(u_i))
\end{align} 
where $P_j(x)=a_jx^j+ \dots$ is any polynomial of order $j$ with leading order coefficient $a_j$. Using orthogonal polynomials for $P_j$, it can be seen that the partition function for a $\beta=2$ ensemble drastically simplifies essentially to the product of the leading order coefficients. To see this, one expands the two determinants in the $\beta=2$ partition function and then makes use of Definition \ref{defi:ortho_pol}, which implies that all the cross terms of two determinants vanish \cite{Me04}.
In our case, we can simply read off the leading order coefficients from the Stieltjes-Wigert polynomials in Lemma \ref{Lem:SW} and use this trick, which turns out to give the following.

\begin{prop}
\label{thm:exactZ_2}
The partition function $Z_2$ for $\mathcal{P}_{\mathbb{C}}(N)$ is given by
\begin{align*}
Z_2(\sigma)= \frac{\omega_2(N)}{2^{N^2}}(2 \pi \sigma^2)^{\frac{N}{2}} \exp \Big((N^3-N)\frac{\sigma^2}{6}\Big)\prod \limits_{k=1}^{N-1}(1-e^{-k\sigma^2})^{N-k}
\end{align*}
\end{prop}

The use of Stieltjes-Wiegert polynomials to solve $Z_2$ in the same way as above has already been known in the context of Chern-Simons theory (see \cite{Ma05}, \cite{Ti04}). Note that Proposition \ref{thm:exactZ_2} relies on the "coincidence" that the orthogonal polynomials are known.

\subsection{Exact formulas for $Z_1$ and $Z_S$}

For technical simplicity we will always assume $N=2m$ is even when we are dealing with $Z_1$ and $Z_S$. For applications this is no major restriction, as we can always simply delete one column/row of data, which is not too significant when we are working with high-dimensional data. In \cite{Ti20}, Section III and Appendix C the case in which $N$ is odd has also been considered, which is an important result in a systematic treatment.

Unfortunately, we do not know of any formulas for skew-orthogonal polynomials with potentials $V_{SW}$ or $V_S$ and it seems hard to reconstruct them from the orthogonal polynomials along the lines of \cite{Ad99} in cases where the potential is not a polynomial. It is worth noting that we do not even know of a formula for orthogonal polynomials with potential $V_S$, which would already be useful (see Section \ref{sec:large_N}).
If we knew of such skew-orthogonal polynomials, we could get similar formulas for $Z_1$ and $Z_S$ by making use of the following De Brujin identities \cite{DB55}:

\begin{lem}[De Brujin Identity]
\label{lem:DeBrujin}
Let $\phi:[a,b] \rightarrow \mathbb{R}^N$ be an integrable function. Then
\begin{align}
\label{eq:DeBrujin}
    \int_{\{a\leq x_1<\dots < x_N\leq b\}} \det_{1\leq i,j\leq N}(\phi_i(x_j))dx=\Pfaff(A)
\end{align}
where \begin{align}
    A_{i,j}=\int_a^b\int_a^b  \operatorname{sign}(x-y )\phi_i(x)\phi_j(y) dx dy.
\end{align}
and $\Pfaff$ denotes the Pfaffian of an antisymmetric matrix \cite{DB55}.
\end{lem}
In fact there is another De Brujin identity that is relevant to the treatment of the quaternionic case $Z_4$ (see III-C of \cite{Ti20}). We will not go through the computational details of this case. Further discussion of this case can be found in \cite{Ti20}, \cite{MM91}.
Applying Lemma \ref{lem:DeBrujin} to our case of interest leads to the following proposition.

\begin{prop}
\label{thm:finiteN_SPD}
Let $N=2m$ be even for technical convenience. Let $a_i$ and $b_i$ be the leading order coefficients of skew-orthogonal polynomials $\{P_i(x; \sigma)=a_i(\sigma)x^i|i=1,\dots ,N\}$ for the potential $V_{SW}(x;\sigma)=\frac{1}{2\sigma^2}\log^2(x)$ on $\mathbb{R}_{>0}$ and  skew-orthogonal polynomials $\{Q_i(x; \sigma)=b_i(\sigma)x^i|i=1,\dots ,N\}$ for the potential $V_{S}(x; \sigma)=\frac{1}{8\sigma^2}\log^2(x+\sqrt{x^2-1})$ on $(1, \infty)$. 

\begin{enumerate}
    \item 
The partition function $Z_1$ is given by \begin{align}
    Z_1(\sigma)=\frac{\omega_1(N)}{2^{Nm}} \exp\big(-N((N-1)/2+1)^2(\sigma^2/2)\big) \bigg (\prod_{l=1}^N a_l(\sigma) \bigg)^{-1}
\end{align}
\item The partition function $Z_S$ is given by
\begin{align}
    Z_S(\sigma)=\vol (\U(N))2^{m^2(N+1)}\bigg (\prod_{l=1}^N b_l(\sigma)(\sigma) \bigg)^{-1}.
\end{align}
A formula for $\vol(\U(N))$ can be found in Lemma \ref{lem:omegas} of the appendix.
\end{enumerate}

\end{prop}
\begin{proof}
\begin{enumerate}
    \item 

First, let us consider $Z_1$. Recall from Equation \eqref{al:Z_beta} that
\begin{align}
   Z_1(\sigma)=  \frac{C_{N,1}}{(2\pi)^N}(\sigma) \int_{\mathbb{R}^N_{>0}}\left(\prod_{i=1}^N \exp\left(-V_{SW}(u_i; \sigma)\right)\right) |\Delta(u)| \prod_{i=1}^N du_i.
\end{align}
The Vandermonde matrix can be rewritten as
\begin{align*}
    \Delta(x)=\det(x_i^{j-1})=\big(\prod_{l=0}^{N-1} c_l\big)^{-1}\det(R_{j-1}(x_i))
\end{align*}
where $R_k(x)=c_kx^k+\dots$ is any degree $k$ polynomial. We will choose $R_k=P_k$ to be skew-orthogonal polynomials so that Definition \ref{defi:ortho_pol} implies that the right hand side of the De Brujin identity \eqref{eq:DeBrujin} reduces to $2^{N/2}=\Pfaff(B)$ where $B$ is now the $2N\times 2N$ block diagonal matrix with $N$ blocks of the form
$\bigl(\begin{smallmatrix}
   0 & 2 \\
   -2 & 0 \\
 \end{smallmatrix}\bigr)$ on the diagonal. Combined with the De Brujin identity \eqref{eq:DeBrujin} this gives: \begin{align}
    Z_1(\sigma)&=\frac{C_{N,1}(\sigma)}{(2\pi)^N} \int_{0< u_1< \dots< u_n<\infty} \det(\exp\left(-V_{SW}(u_i; \sigma)\right)u_i^{j-1})\prod_{i=1}^N du_i\\
    &=\frac{C_{N,1}(\sigma) }{(2\pi)^N}\int_{0< u_1< \dots< u_n<\infty} \det(\exp\left(-V_{SW}(u_i; \sigma)\right)u_i^{j-1})\prod_{i=1}^N du_i
    \\
    &\overset{\eqref{eq:DeBrujin}}{=}\frac{\omega_1(N)}{2^{NN_1}} \exp\big(-NN_1^2(\sigma^2/2)\big) \big(\prod_{l=0}^{N-1} a_l\big)^{-1} 2^m \\
    &=\frac{\omega_1(N)}{2^{Nm}} \exp\big(-N((N-1)/2+1)^2(\sigma^2/2)\big) \bigg (\prod_{l=1}^N a_l \bigg)^{-1}
\end{align}
where we have used $N_1=(N-1)/2+1$ and $NN_1-m=Nm$.
\item
The same trick works for $Z_S$. Namely, start from Equation \eqref{eq:Z_S} and rewrite the Vandermonde determinant now in terms of $Q_i$. Again, using the De Brujin identity \eqref{eq:DeBrujin}, we arrive at
\begin{align}
      Z_S(\sigma)&=\vol (\U(N))2^{\frac{N(N+1)}{2}}\int_{1<u_1<\dots<u_N<\infty}
\prod \limits_{i=1}^N \exp \big(- V_{S}(u_i)\big) \Delta(u) \prod \limits_{i=1}^N du_i\\
&=\vol (\U(N))2^{m^2(N+1)}\bigg (\prod_{l=1}^N b_l \bigg)^{-1}.
\end{align}

\end{enumerate}
\end{proof}

\begin{rem}
\label{rem:fin_N}
In \cite{Ti20}, the De Brujin identity is used to numerically compute $Z_1$ directly from the Pfaffian without skew-orthogonal polynomials. This seems to get numerically challenging for the orders of $N$ we would need for many applications, e.g. $N > 100$.
An alternative approach may be to compute the skew-orthogonal polynomials numerically via a symplectic Gram-Schmidt algorithm and then use Proposition \ref{thm:finiteN_SPD}. This might still be numerically hard as we would have to find different polynomials for every $\sigma$ independently. Nonetheless, Proposition \ref{thm:finiteN_SPD} is still very useful, as it can give us some intuition about when the approximation by the large $N$ limit, which will be discussed in Section \ref{sec:large_N}, is good enough. 
\end{rem}
\begin{rem}
In Appendix A of \cite{Ti20}, the first three examples of skew-orthogonal polynomials have been computed explicitly. Moreover, \cite{Ti20} relates these skew-orthogonal polynomials to a very important object that we do not study in great detail here: The \textbf{expected} eigenvalue density $\rho$ at finite $N$. We do not discuss this function for finite $N$ and refer to \cite{Ti20} (chapters II, III) for a discussion where it is studied for $\beta \in \{1,2,4\}$ and also for $N$ even and odd. We will discuss the large $N$ limit of this distribution in the next section.
\end{rem}

\section{The large $N$ Limit}
\label{sec:large_N}

Motivated by the applications from Section \ref{subsec:Moti} and the difficulties  of solving $Z_1$ and $Z_S$ for finite, large $N$ (see Remark \ref{rem:fin_N}), we are interested in the \textbf{large $N$} or \textbf{planar limit}  where $N\rightarrow \infty$ while the so-called \textbf{'t Hooft parameter} $t=N\sigma^2$ is fixed. In particular, this limit is about more than just sending $N\rightarrow \infty$, so that the term large $N$ limit can be a bit misleading. This limit is frequently considered in various situations in the physics literature originally dating back to the 1970s where it was studied in the context of quantum chromodynamics (see \cite{tH74}). A perspective on its physical interpretation and origin will be given in Section \ref{subsubsec:QFT}.
As we are keeping $t=N\sigma^2$ fixed while sending $N\rightarrow \infty$, we also have $\sigma^2\rightarrow0$. For this reason, the limit is referred to as the \textbf{double scaling limit} in section II, 3. c) of \cite{Ti20}. \cite{Ti20} also studies other interesting limits  in II, 3. a), b), namely $\sigma^2 \rightarrow 0$ and $\sigma^2 \rightarrow \infty$ while $N$ is fixed.

\subsection{The large $N$ limit from orthogonal polynomials}

\label{subsct:limit_ortho}
We begin with an elementary calculation of the large $N$ limit for $\beta=2$ starting from Proposition \ref{thm:exactZ_2}. 

\begin{prop}
\label{thm:Z_2_calc}
The large $N$ limit of $Z_2$ is given by \begin{align}
\label{eq:calculatian_Z_2}
    \frac{1}{N^2}\log (Z_2(\sigma))\sim - \frac{1}{2} \log \left(\frac{2N}{\pi}\right)+ \frac{3}{4}+ \frac{t}{6} - \frac{\Li_3(e^{-t})-\zeta(3)}{t^2}
\end{align}
where $\Li_3(x)\coloneqq \sum \limits_{k=1}^\infty \frac{x^k}{k^3}$ (for $|x|<1$) is the trilogarithm and $\zeta$ is the Riemann Zeta function. 

\end{prop}

The proposition follows from a direct calculation based on the following lemmas.

\begin{lem}
For the prefactor, we have \begin{align}
    \frac{1}{N^2} \log(\omega_2(N)) \sim - \frac{1}{2} \log\left( \frac{N}{2\pi}\right) + \frac{3}{4}.
\end{align}
\end{lem}
This is further discussed in the Appendix.
\begin{lem}
In the large $N$ limit, we have the following asymptotics \begin{align}
\label{equ:trilog}
    \frac{1}{N^2} \log\left(\prod \limits_{k=1}^{N-1}(1-e^{-k\sigma^2})^{N-k}\right) \sim \int_0^1 (1-x)\log(1-e^{-tx})dx.
\end{align}
The improper integral on the right-hand side of Equation \eqref{equ:trilog} equals $-\frac{\Li_3(e^{-t})- \zeta(3)}{t^2}$
\end{lem}
\begin{proof}
Taking the logarithm of the product, the left-hand side of Equation \eqref{equ:trilog} reads \begin{align}
    \frac{1}{N^2} \log \left(\prod \limits_{k=1}^{N-1}(1-e^{-k\sigma^2})^{N-k}\right)= \frac{1}{N} \sum \limits_{k=1}^{N-1} (1- \frac{k}{N}) \log(1-e^{-t \frac{k}{N}})
\end{align}
This is a Riemann sum for the improper integral on the right-hand side. To evaluate this integral, one may introduce the power series of the logarithm under the integral \begin{align}
    \int_0^1(1-x)\log(1-e^{-tx})dx = - \sum \limits_{k=1}^{\infty} \frac{1}{k} \int_0^1 (1-x)e^{-ktx} dx
\end{align} 
and note that \begin{align}
    \int_0^1 (1-x)e^{-ktx}dx= \frac{1-e^{-ktx}}{(kt)^2}
\end{align}
in order to obtain  $-\frac{\Li_3(e^{-t})- \zeta(3)}{t^2}$.
\end{proof}
\begin{rem}
Note that we can drop the strange term $\log\left(\frac{N}{\pi}\right)$ of Equation \eqref{eq:calculatian_Z_2} by considering an overall $\sigma$-independent rescaling of $Z_2$:
\begin{align}
    \frac{1}{N^2}\log \left(\frac{Z_2(\sigma)}{\omega_2(N)}\right)\sim - \frac{1}{2} \log \left(\frac{2N}{\pi}\right)+ \frac{3}{4}+ \frac{t}{6} - \frac{\Li_3(e^{-t})-\zeta(3)}{t^2}.
\end{align}
Such a rescaling has no impact on any applications to parameter estimation, since there, one only considers derivatives of $\log(Z_2)$ with respect to $\sigma$ \cite{Sa17} as also noted in Section \ref{subsec:Moti}.
\end{rem}

\subsection{The saddle-point perspective}
The calculation in Section 
\ref{subsct:limit_ortho} crucially relies on having solved the partition function for any finite $N$ in the first place. Since this is not the case for $Z_1$ and $Z_S$ as discussed in Remark \ref{rem:fin_N}, we will now give a more powerful approach to the large $N$ limit inspired by the theoretical physics literature such as \cite{Ti04}, \cite{Ma05}. This approach essentially boils down to a saddle-point approximation \cite{Ma05}.

\subsubsection{The QFT setup and planar Feynman diagrams}
\label{subsubsec:QFT}

In this section, we interpret $Z$ as being the partition function of a  $0$-dimensional toy-model quantum field theory (QFT) and then discuss how perturbative arguments from physics lead to further insights. Although the physical arguments can be formalized in this $0$-dimensional setting, we will not aim to do so, as we include this section mostly for motivational purposes. The main goal is to give a physical argument for how the large $N$ limit naturally appears in QFT and how it can be interpreted in terms of Feynman diagrams. This section may well be skipped by the reader only interested in the calculation of the large $N$ limit of $Z_{\beta}$ and $Z_S$. A background on general and $0$-dimensional QFTs to the level we will need can be found in chapters 8.1, 8.2, 9.1 and 36.2 of \cite{Ho03}.

\begin{rem}
In the language of quantum field theory, the partition function $Z_{\beta}$ of covariance matrices $\mathcal{P}_{\mathbb{K}}(N)$ in the form of \eqref{al:Z_beta} can be viewed as the partition function of a quantum field theory with the following data:
\begin{enumerate}
    \item A point $\{\star\}$ as the \textbf{spacetime} of the theory.
    \item The field content consisting of a map $\phi: \{\star \}\rightarrow \mathbb{R}^N$ so that the space of fields (over which is integrated in the path integral) is simply $\mathbb{R}^N$.
    \item An action given by $S(\phi)\coloneqq S(\phi(\star))=\frac{1}{2\sigma^2}\log^2(\phi)$. The $\sigma^2$ is to be interpreted as some sort of Planck constant $\hbar$ and we include it into the action just for notational convenience.
    \item A path integral measure, $\mathcal{D}\phi= |\prod_{i<j}(\phi_i-\phi_j)|^{\beta}d^n\phi$.
\end{enumerate}
In this QFT picture, Equation \eqref{al:Z_beta} arises after fixing the gauge of a zero-dimensional field theory. Specifically, the role of the gauge group is played by the group $H$, which gets integrated out of the general formula (\ref{eq:Z_general}). The field takes values in $\mathcal{P}_{\mathbb{K}}$, and then the Vandermonde determinant can then be viewed as Faddeev-Popov determinant, coming from the standard Faddeev-Popov gauge-fixing procedure. This is discussed in (1.33) of \cite{Ma05}. As the action includes higher order terms than just quadratic, we are in fact dealing with an \textbf{interacting QFT}.
 
 Viewing our theory as a gauge theory is actually the right perspective since we will make use of certain types of Feynman diagrams, which have double-lined edges and are characteristic to gauge theories. As we will not go into any details here anyway, we will not worry about this. An elaborated review can be found in Section 1.1 of \cite{Ma05}.
 \end{rem}

Now, that we can view $Z$ as the partition function of a $0$-dimensional gauge theory, we can import some physical insights into our picture. First of all, define the \textbf{free energy} by 
\begin{align}
F(\sigma)= \log(Z(\sigma)).
\end{align}
It is a common theme in perturbative QFT to expand the free energy in certain parameters (coupling constants and $\hbar=\sigma^2$) and interpret every term pictorially as a so-called \textbf{Feynman diagram}. In our case, we do not have any coupling constants, which anyway will not affect the interpretation of $\sigma^2$ as $\hbar$. The Feynman diagrams are built up from double lined edges (also called \textbf{propagators} in the physics literature) as in Figure \ref{fig:feynman} (a) and vertices as in Figure \ref{fig:feynman} (b) for an example of a quartic vertex. The fact that we have two lines on each edge comes from integrating over matrices, which have $2$ indices in the gauge theory perspective (see 1.1 of \cite{Ma05}). These kinds of Feynman diagrams are often referred to as \textbf{ribbon diagrams} or \textbf{fatgraphs} and two examples can be seen in Figure \ref{fig:feynman} (c).

\begin{figure}
\centering
\includegraphics[width=1\linewidth]{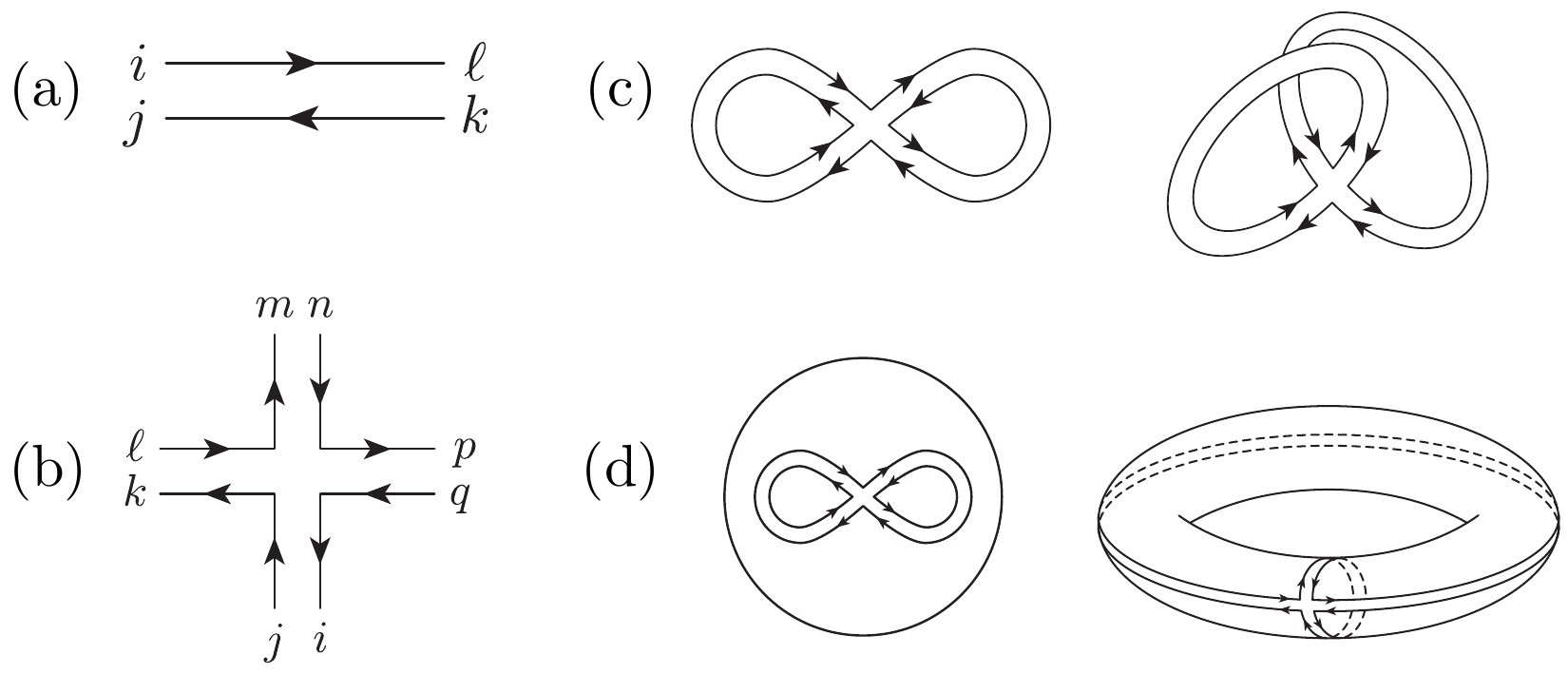}
\caption{(a) An edge of a Feynman diagram. (b) A quartic vertex of a Feynman diagram. As we have higher order terms (infinitely many) in our action, a general Feynman diagram will also have cubic, quintic, and higher order vertices. (c) Two Feynman diagrams that can be constructed from these edges and vertices. (d) Embedding of the Feynman diagrams of (c) in $S^2=\Sigma_0$ with $(g,h)=(0,3)$ (left) and in a genus-1 surface $T^2=\Sigma_1$ with $(g,h)=(1,1)$ (right), respectively.}
\label{fig:feynman}
\end{figure}

A Feynman diagram with $E$ edges, $V$ vertices and $h$ boundary components is of the order $(\sigma^2)^{E-V}N^h$ since it can be seen (by the defining Feynman rules) that each edge gives a power of $\sigma^2$ and each vertex gives a power of $\sigma^{-2}$. Most importantly, each boundary component gives a power of $N$ as such a closed loop amounts to some index contraction of the form 
\begin{align}
\sum \limits_{j=1}^k\sum \limits_{i_j=1}^N \delta_{i_1}^{i_2}\delta_{i_2}^{i_3}\dots \delta_{i_{k-1}}^{i_k} \delta_{i_k}^{i_1}=\sum_{i=1}^N \delta_i^i=N.
\end{align} 
Now, by elementary properties of the Euler characteristic (see \cite{Ha02}, page 146-147), we have $2-2g=\chi(\Sigma_g)= V-E+h$, where $g$ is the genus of the surface $\Sigma_g$ obtained by collapsing every boundary component of our Feynman diagram to a point or equivalently it is the minimal genus of a surface in which the Feynman graph can be embedded. See Figure \ref{fig:feynman} (d) for an example. So, the contribution of Feynman diagram with $E$ edges, $V$ vertices and $h$ boundary components depends on $N$ and $\sigma^2$ only through \begin{align}
    (\sigma^2)^{E-V}N^h=(\sigma^2)^{2g-2+h}N^h=t^{2g-2+h}N^{2-2g}.
\end{align}

Finally, the free energy is given by summing over all the possible connected Feynman diagrams (actually, we only sum over so-called \textbf{Vacuum bubbles}, which are diagrams without any external propagators) and grouping them in terms of $g$ and $h$, we get the following \textbf{genus expansion}: \begin{align}
\label{eq:free_energy}
    F(\sigma^2)=\sum \limits_{g=0}^{\infty} \sum \limits_{h=1}^{\infty} F_{g,h}t^{2g-2+h}N^{2-2g}=\sum \limits_{g=0}^{\infty} f_g(t) N^{2-2g}.
\end{align}

\begin{rem}
Equation \eqref{eq:free_energy} gives a beautiful interpretation of the large $N$ limit: The free energy (respectively the partition function) consists of a sum of infinitely many terms each corresponding to a Feynman diagram, that is labelled by its genus $g$ and the number of boundary components $h$ (e.g. see Figure \ref{fig:feynman} (d)). Every such diagram corresponds to a term depending on $t$, $\sigma^2$ and $N$ only through $t^{2g-2+h}$ and $N^{2-2g}$. Now, in the large $N$ limit where $N\rightarrow \infty$ while $t$ is fixed, we see that all the diagrams of genus $g>0$ (e.g. Figure \ref{fig:feynman}, (d) (right)) get suppressed and only the diagrams with $g=0$ contribute. Such diagrams, so-called \textbf{planar diagrams}, can be embedded into the $2$-sphere and hence can be drawn on a plane (see Figure \ref{fig:feynman}, (d) (left)). This also motivates calling $N\rightarrow \infty$ the planar limit.
\end{rem}

As there are still infinitely many Feynman diagrams in $f_0(t)$ and there are no coupling constants that we can send to zero, to further reduce $f_0(t)$ to finitely many diagrams, we will not attempt to compute the large $N$ limit directly from Feynman diagrams.
Rather, we will find the large $N$ limit by solving a saddle-point equation in the next section. This saddle-point equation comes from the perspective that $\sigma^2$ should be interpreted as $\hbar$, so that after fixing $t=N\sigma^2$, $\frac{1}{N}$ is proportional to $\hbar$. Hence, the large $N$ limit is a semiclassical limit with $f_0$ being the classical contribution, which can be derived from extremizing the action $S$. Anyway, we note that the Feynman diagram interpretation is certainly an interesting non-trivial perspective that might lead to more than just the above qualitative insight in the given context.

\subsubsection{The saddle-point equation}
In Proposition \ref{thm:largeN}, we will see how to directly obtain $f_0$ from solving a certain singular integral equation, the \textbf{saddle-point equation}.
Moreover, the solution to this saddle-point equation is essentially independent of $\beta>0$ (up to a rescaling $t \mapsto \beta t$). This means in particular that we can import our calculation of the large $N$ limit of $Z_2$ from Section \ref{subsct:limit_ortho} to compute the large $N$ limit of $Z_1$ (or $Z_4$).

As above, we split  $Z_{\beta}(\sigma)=C_{N,\beta}(\sigma)\Tilde{Z}_{\beta}(\sigma)$ with the prefactor
\begin{align}
    C_{N,\beta}(\sigma)=\frac{\omega_{\beta}(N)(2\pi)^N}{2^{NN_\beta}}\exp(-NN^2_{\beta} \frac{\sigma^2}{2}).
\end{align}
The partition function $\Tilde{Z}_{\beta}$ is in an appropriate form for the use of
saddle-point methods as discussed in \cite{Ma05} (see Equation (1.46) therein):
\begin{align}
\label{eq:ZTilde}
 \Tilde{Z}_{\beta}(\sigma)=\frac{1}{N!}\int_{\mathbb{R}^N_+} \exp \big(N^2 \Tilde{V}_{SW}(\vec{\lambda},\beta;t) \big)\prod_{k=1}^N \frac{d\lambda_k}{2\pi},
\end{align}
where we have introduced the \textbf{effective potential}:
\begin{align}
\label{eq:S_eff}
    \Tilde{V}_{SW}\left(\vec{\lambda},\beta;t\right)\coloneqq-\frac{1}{2tN}\sum_{i=1}^N \log^2(\lambda_i)+ \frac{\beta}{N^2} \sum_{i<j} \log |\lambda_i-\lambda_j|.
\end{align}
In the large $N$ limit, the dominant contributions to $\Tilde{Z}_{\beta}$ will come from the saddle points of $\Tilde{V}_{SW}$ and the saddle-point equation below will simply characterise the saddle points (in a certain continuum limit). For the Siegel domain, we have a similar effective potential, which takes the form
\begin{align}
\label{V_S_eff}
    \Tilde{V}_{S}\left(\vec{\lambda}; t\right) \coloneqq-\frac{1}{8tN}\sum_{i=1}^N \log^2\left(\lambda_i+\sqrt{\lambda_i^2-1}\right)+ \frac{1}{N^2} \sum_{i<j} \log |\lambda_i-\lambda_j|.
\end{align}

\begin{rem}
\label{rem:semicirc}
 The effective potential $\Tilde{V}_{SW}$ gives rise to a new physical interpretation of the theory and its large $N$ limit: the $N$ eigenvalues can be seen as static particles in the potential $V_{SW}$ interacting through a logarithmic Coulomb repulsion (the second term of the effective potential). As $\sigma^2=\frac{t}{N}$ decreases, the repulsion becomes weak and all particles can sit next to each other close to the minimum of the potential ($\lambda=1$), while the particles tend to spread out for large $\sigma^2$ as observed in Figure \ref{fig:plots}. Now, the large $N$ limit can be seen to correspond to the addition of more and more particles (i.e. $N\rightarrow \infty$) while letting their repulsion become increasingly weak (fixing $t=N\sigma^2$). Finding the limiting distribution $\rho_t$ (that still depends on $t$) turns out to characterize the large $N$ limit of the partition function. It can be obtained by solving the saddle-point equation for $\Tilde{V}_SW$ in the continuum limit as discussed below.
 Motivated by the physics literature, we will refer to this $\rho_t$ as the \textbf{master field}. 

\end{rem}

\begin{rem}
\label{rem:universality}
A further observation, from the effective action \eqref{eq:S_eff} is that up to an overall factor of $\Tilde{V}$ (which leaves the saddle-point equation invariant), we can absorb the $\beta$ by rescaling $t \mapsto t/\beta$. So, in the large $N$ limit, $\Tilde{Z}_{\beta}(\sigma) \sim \Tilde{Z}_{2}(\sqrt{\beta/2}\;\sigma)$ irrespective of the choice of $\beta \in \{1,2,4\}$, which is remarkable considering the quite distinct geometric origins associated with the different values of $\beta$.
 We refer to this phenomenon as \textbf{universality}: the three cases $\mathbb{K} \in \{\mathbb{R}, \mathbb{C}, \mathbb{H}\}$ (respectively $\beta \in \{1,2,4\}$) started out differently, but "flow" to a universal limit characterized by the master field as $N\rightarrow \infty$. This is specially useful for the particularly important case of $\beta=1$, which can now be related to the large $N$ limit of $Z_2$ derived in Proposition \ref{thm:Z_2_calc}:
\begin{align}
    \frac{1}{N^2} \log (Z_{\beta}(\sigma))\sim \frac{1}{N^2}\left[ \log\left(Z_2\left( \sqrt{\tfrac{\beta}{2}} \sigma\right)\right) - \log\left(\frac{C_{\infty, \beta}(\sigma)}{C_{\infty,2}(\sqrt{\tfrac{\beta}{2}}\sigma)}\right)\right]
\end{align}
where $C_{\infty, \beta}$ are the large $N$ limits of the prefactors, discussed in Appendix \ref{subsec:prefactors}. We will go through this in detail in Proposition \ref{thm:Z_beta_scale}.
Below, we will give a further formula for the large $N$ limit of $Z_\beta$, by solving the saddle-point equation explicitly. 
\end{rem}

The physical interpretation of the large $N$ limit from Remark \ref{rem:semicirc} might be more familiar than the Feynman diagrams, since it is well known in the random matrix theory literature as the \textbf{Coulomb gas method} \cite{Me04}. The most prominent example where it is usually introduced is in the derivation of \textbf{Wigner's semicircle law}.

\begin{ex}[Wigner's semicircle law]
\label{ex:Wig}
For simplicity, since $V_Q(\lambda;\sigma)\coloneqq \frac{1}{2\sigma^2} \lambda^2$ has a simpler form than $V_{SW}$ or $V_S$ yet illustrates all the necessary ideas, we will begin by considering its saddle-point equation.  The three cases $\beta \in \{1,2,4\}$ are now known as orthogonal, unitary and symplectic ensembles in classical random matrix theory \cite{Me04}. As motivated in Remark \ref{rem:semicirc} and following  \cite{Ma05} ((1.47)-(1.53)), we are interested in the saddle points of the effective potential, which in this case reads \begin{align}
    \Tilde{V}_Q(\vec{\lambda};t)\coloneqq-\frac{1}{tN} \sum_{i} \lambda_i^2 + \frac{\beta}{N^2} \sum_{i<j} \log |\lambda_i-\lambda_j|
\end{align}
in analogy with equations \eqref{eq:S_eff} and \eqref{V_S_eff}. The saddle points are characterized by $\frac{d}{d\lambda_k}\Tilde{V}_Q\left(\vec{\lambda};t\right)=0$ for all $k = 1, \dots, N$, which is simply \begin{align}
\label{eq:fin_saddle}
     \frac{1}{\beta t} \lambda_k=\frac{1}{N}\sum_{j\neq k} \frac{1}{\lambda_k-\lambda_j}=:P\left(\int \frac{\rho_{t,N}(\lambda)}{\lambda-\lambda_k}d \lambda \right) 
\end{align}
where $\rho_{t,N}$ is formally given by $\rho_{t,N}(\lambda)=\frac{1}{N} \sum_j \delta(\lambda-\lambda_j)$ and $P$ is a discrete Cauchy principal value.
 The large $N$ limit can now be regarded as a continuum limit, in which $\rho_{t,N}$ becomes a continuous function $\rho_t$. Equation \eqref{eq:fin_saddle} becomes 
\begin{align}
\label{eq:eom}
       \frac{1}{\beta t} \lambda= P\left(\int_{-\infty}^{\infty} \frac{\rho_t(\lambda')d \lambda'}{\lambda-\lambda'} \right)
       \iff      \frac{1}{2 t} \lambda= P\left(\int_{-\infty}^{\infty} \frac{\rho_{\frac{2}{\beta}t}(\lambda')d \lambda'}{\lambda-\lambda'} \right) 
\end{align}
where now $P$ is the actual Cauchy principal value. The \textbf{saddle-point equation} \eqref{eq:eom} can be solved using resolvent methods and in \cite{Ma05} (Equation (1.82)) it is 
shown that the solution in this case is simply a semicircle of radius $2 \sqrt{t}$: 
\begin{equation}
\rho_t(\lambda)=\frac{1}{2\pi t} \sqrt{4t-\lambda^2} \chi_{\mathcal{C}(t)}(\lambda)
\end{equation}
where $\chi_{\mathcal{C}(t)}$ is the characteristic function supported on the interval 
$\mathcal{C}(t)=[-2\sqrt{t},2 \sqrt{t}]$. 
This celebrated semicircle law was first derived by Wigner in 1955 \cite{Wi55}.

Figure \ref{fig:plots} shows the master field for $t=\frac{1}{4}$ and $\beta \in \{1,2,4\}$. We observe the expected universal scaling behaviour discussed in Remark \ref{rem:semicirc}. It has even been argued that the semicircle law remains universal, in cases where the entries of the matrix are not independent anymore, which goes beyond the normal $\beta$-ensembles  (see \cite{Kra17}).
Moreover, we see that $t \rightarrow 0$ corresponds to $\rho_t$ being supported only in a small interval around $0$ as expected from the interpretation in Remark \ref{rem:semicirc} and since $0$ is the minimum of the potential $V_Q$.
\end{ex}

The following proposition will be devoted to discussing similar steps for the potential $V_{SW}$.

\begin{prop}
\label{thm:largeN}
Let $\beta>0$. If $N\rightarrow\infty$ while the 't Hooft parameter $t=N\sigma^2$ is fixed, we get \begin{align}
    \frac{1}{N^2}\log( \Tilde{Z}_{\beta}(\sigma))\sim F_{uni}\left(\frac{\beta}{2} t\right)+ \mathcal{O}(N^{-2})
\end{align}
where \begin{align}
    F_{uni}(t)=-\frac{1}{2 t}\int_{\mathcal{C}(t)}\rho_t^{SW} (\lambda) \log^2\lambda \; d\lambda+ \int_{\mathcal{C}(t)^2} \rho_t^{SW}(\lambda) \rho_t^{SW} (\lambda') \log(|\lambda-\lambda'|) d\lambda d \lambda'
\end{align}
and the master field $\rho_t^{SW}$ is given by 
\begin{align}
\label{eq:rho}
   \rho_t^{SW} (\lambda)&= \frac{1}{\pi t \lambda} \tan^{-1}\left[ \frac{\sqrt{4 \lambda-(1+e^{-t}\lambda})^2}{1+e^{-t}\lambda} \right] \chi_{\mathcal{C}(t)}
\end{align}  
where $\mathcal{C}(t)=[2e^{2t}-e^t+2e^{\frac{3t}{2}} \sqrt{e^t-1},2e^{2t}-e^t-2e^{\frac{3t}{2}} \sqrt{e^t-1}]$.
\end{prop}

\begin{rem}
In Figure \ref{fig:plots}, Wigner's semicircle distribution is plotted alongside $\rho_t^{SW}$ for different choices of $\beta\in\{1,2,4\}$. It can be observed that as $t$ decreases (or equivalently, $\beta$ decreases), the distibutions tend to concentrate around the classical minima $\lambda=0$ (for $V_Q$) and $\lambda=1$ (for $V_{SW}$). Conversely, they tend to spread out as $t$ increases. 
The master field $\rho_t^{SW}$ which is the content of the above result has in fact previously appeared in the physics literature (see \cite{Ma05}, (2.194) on page 64) and less directly in the random matrix theory literature \cite{CL98}, \cite{KA99}. There, it has been derived using resolvent techniques which we will not review here.

\end{rem}

\begin{proof}
This proof will be quite informal and as above, we will not focus on the analytical challenges associated with taking continuum limits.
Recall, that \begin{align}
    \Tilde{Z}_{\beta}(\sigma)=\frac{1}{N!}\int_{\mathbb{R}^N_+} \exp \big(N^2 \Tilde{V}_{SW}(\lambda,\beta;t) \big)\prod_{k=1}^N \frac{d\lambda_k}{2\pi}
\end{align}
where in Equation \eqref{eq:S_eff}, we defined \begin{align}
\label{eq:proof_s_eff}
   \Tilde{V}_{SW}(\lambda,\beta;t) &= -\frac{1}{2tN}\sum_{i=1}^N \log^2(\lambda_i)+ \frac{\beta}{N^2} \sum_{i<j} \log |\lambda_i-\lambda_j|\\
    &=-\frac{1}{2 t}\int_{\mathbb{R}_{>0}}\rho_{t,N} (\lambda) \log^2(\lambda) d\lambda
    +\beta \int_{\mathbb{R}_{>0}^2} \rho_{t,N} (\lambda) \rho_{t,N}  (\lambda') \log(|\lambda-\lambda'|) d\lambda d \lambda'.
\end{align}
Here $\rho_{t,N}^{SW}$ is the density of eigenvalues of a finite-dimensional matrix, formally given by $\rho_{t,N}^{SW}(\lambda)=\frac{1}{N} \sum \limits_{j=1}^N \delta(\lambda-\lambda_j)$ as in Example \ref{ex:Wig}. 

Again, following \cite{Ma05} (1.47)-(1.53), the leading contribution to $\Tilde{Z}_{\beta}(\sigma)$ in the large $N$ limit is characterized by the saddle point of the effective potential which obeys $\frac{d}{d\lambda_k}\Tilde{V_{SW}(\lambda,\beta;t)}=0$ for all $k\in \{1, \dots, N\}$. In parallel to Example \ref{ex:Wig}, this saddle-point equation becomes a singular integral kernel equation for the master field $\rho_t^{SW}$ in the large $N$ limit:
\begin{align}
\label{eq:eom}
       \frac{1}{\beta t} \frac{\log(\lambda)}{\lambda}= P\left(\int_0^{\infty} \frac{\rho_t^{SW}(\lambda')d \lambda'}{\lambda-\lambda'} \right)\\
       \iff      \frac{1}{2 t} \frac{\log(\lambda)}{\lambda}= P\left(\int_0^{\infty} \frac{\rho_{\frac{2}{\beta}t}(\lambda')d \lambda'}{\lambda-\lambda'} \right) 
\end{align}
Again, this can be solved using resolvent methods (see \cite{Ma05}, (2.186)-(2.194)) and it can be seen that the solution is indeed given by \eqref{eq:rho}.
Finally, just by substituting $\rho_{t}$ for $\rho_{t,N}$ in Equation \eqref{eq:proof_s_eff}, we obtain the effective potential in the large $N$ limit evaluated on the master field. But as discussed previously, this is just the leading order contribution to $\frac{1}{N^2}\log(Z(\sigma))$  (see \cite{Ma05}, (1.54)), i.e. $F_{uni}(t)$ :
\begin{align}
    F_{uni}(t)=-\frac{1}{2 t}\int_{\mathbb{R}_{>0}}\rho_{t} (\lambda) \log^2(\lambda) d\lambda
    + \beta \int_{\mathbb{R}_{>0}^2} \rho_{t} (\lambda) \rho_{t}  (\lambda') \log(|\lambda-\lambda'|) d\lambda d \lambda'
\end{align}
which completes the proof.
\end{proof}

\begin{figure}
\centering

\includegraphics[width=0.9\textwidth]{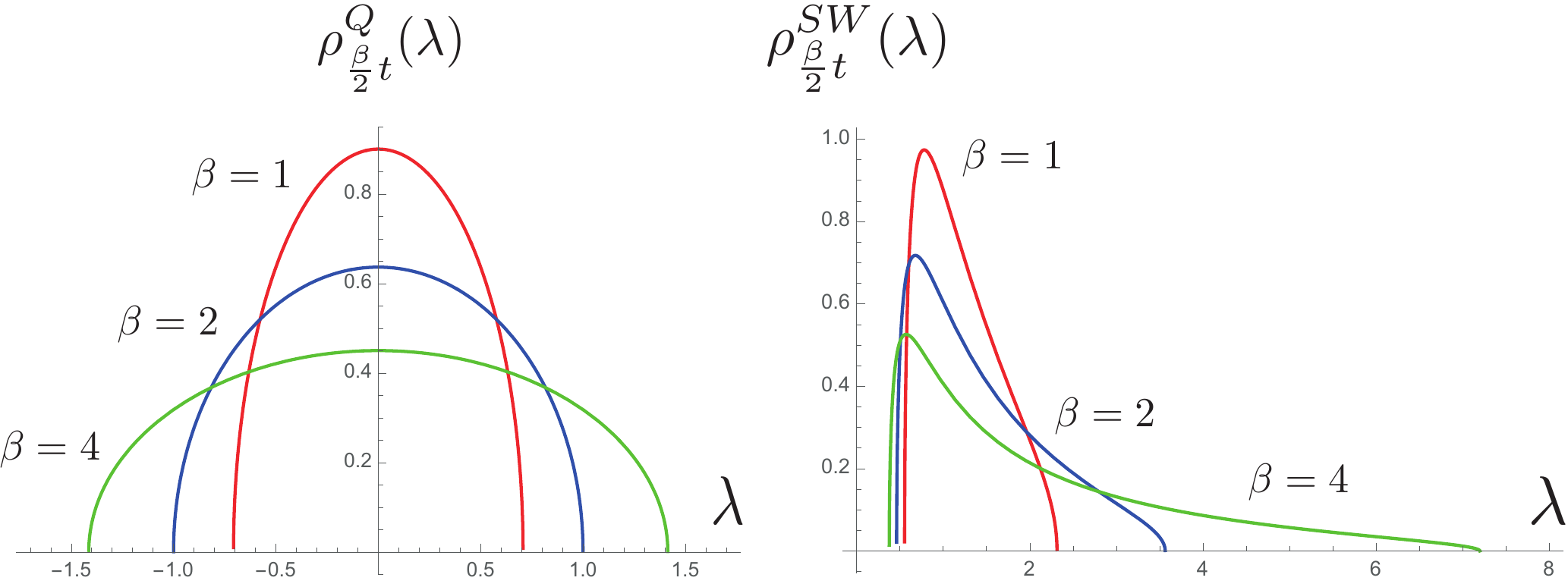}
    
    \caption{The master field $\rho^{Q}_{\beta t/2}$ of Wigner's semicircle law (left) and $\rho_{\beta t/2}^{SW}$ (right) for fixed $t=1/4$ and different choices of $\beta$. The red, blue and green curves correspond to the real, complex and quaternionic cases, respectively ($\beta = 1,2,4$).}

\label{fig:plots}

\end{figure}

\subsubsection{The large $N$ limit for $\beta=1$}
Summing up the previous discussion, the next proposition includes two formulas for the large $N$ limit of $Z_1$.
\begin{prop}
\label{thm:Z_beta_scale}
Let $C_{\infty, 1}(\sigma)$ be the large $N$ limit of the prefactor $C_{N,1}(\sigma)$ as discussed in Lemma \ref{lem:Ccalc}. The large $N$ limit of the partition function $Z_1$ is related in the following way to that of $Z_2$: \begin{align}
    \frac{1}{N^2} \log (Z_{1}(\sigma))\sim \frac{1}{N^2}\left[ \log\left(Z_2\left( \tfrac{1}{\sqrt{2}} \sigma\right)\right) - \log\left(\frac{C_{\infty, 1}(\sigma)}{C_{\infty,2}(\tfrac{1}{\sqrt{2}}\sigma)}\right)\right].
\end{align}
Another formula is given by \begin{align}
\label{eq:Zcoeffplanar}
    \frac{1}{N^2}\log(Z_{1}(\sigma))\sim \frac{1}{N^2} \log(C_{\infty,1}(\sigma))+ F_{uni}\left(\tfrac{1}{2}t\right).
\end{align}
where we follow the notation of Proposition \ref{thm:largeN}.
\end{prop}
\begin{proof}
We discuss the first formula for general $\beta \in \{1,2,4\}$. This is simply the definition \begin{align}
    Z_{\beta}(\sigma)&= C_{N,\beta}(\sigma)\Tilde{Z}_{\beta}(\sigma)\\
    \iff \frac{1}{N^2} \log(Z_{\beta}(\sigma))&= \frac{1}{N^2} \log\left(C_{N,\beta}(\sigma)\right) +\frac{1}{N^2} \log(\Tilde{Z}_{\beta}(\sigma))
\end{align}
combined with Proposition \ref{thm:largeN}, which particularly says \begin{align}
    \frac{1}{N^2}\log(\Tilde{Z}_{\beta}(\sigma))\sim \frac{1}{N^2}\log(\Tilde{Z}_{2}( \sqrt{\frac{\beta}{2}}\sigma)).
\end{align}
Equation \eqref{eq:Zcoeffplanar} is simply the statement of Proposition \ref{thm:largeN} for $\beta=1$ with the large $N$ limit of the prefactor included.
\end{proof}

\subsubsection{The large $N$ limit of the Siegel domain and class DIII}

As in Proposition \ref{thm:largeN}, the large $N$ limit of $Z_S$ is fully characterised by its master field $\rho_t^S$. Computing the derivative $\frac{d}{d\lambda}V_S$, we see that $\rho_t^S$ in turn is governed by the following saddle-point equation:
\begin{align}
\label{eq:saddle_S}
   \frac{\log\left(\lambda+\sqrt{\lambda^2-1}\right)}{4t\sqrt{\lambda^2-1}}=     P\left(\int_1^{\infty} \frac{\rho_t^S(\lambda')d \lambda'}{\lambda-\lambda'} \right).
\end{align}
In this case we do not know of a specific solution to the equation yet. However, we hope to solve equation \eqref{eq:saddle_S} either analytically using resolvent methods or numerically in future work. 

Another way to approach the large $N$ limit of the Siegel domain is to consider the same ensemble with $\beta=2$, which may be solved exactly \cite{Ti20}. From this, we might compute the $\beta=2$ large $N$ limit with potential $V_S$ which would then be related to the $\beta=1$ large $N$ limit (i.e. the large $N$ limit of $Z_S$) in a way similar to propositions \ref{thm:exactZ_2} and    \ref{thm:Z_beta_scale}. 

\begin{rem}
There is another interesting universality here between Equation \eqref{eq:saddle_S} and the saddle-point equation of a partition function 
\begin{align}
z_{\beta=1}^{\mathfrak{so}(2N)}=\frac{1}{2N!} \int_{0<r_1<\dots<r_N}
\prod \limits_{i=1}^N e^{- \frac{r_i^2}{2\sigma^2}} \prod \limits_{i<j} 4\sinh |\frac{r_i-r_j}{2}|\sinh |\frac{r_i+r_j}{2}| \prod \limits_{i=1}^N da_{ii}
\end{align}
introduced and studied in \cite{Ti20}, Appendix C. In the language of \cite{Zi11}, this partition function corresponds to the class $DIII$ as identified by \cite{Ti20}. In \cite{Ti20}, (91)-(94), this partition function has been computed to be 
\begin{align}
    \int_{1<x_N<\dots<x_1}
\prod \limits_{i<j}(x_i-x_i)\prod \limits_{i=1}^N \exp(-V^{so}(x_i;\sigma)) dx_i
\end{align}
with the potential \begin{align}
    V^{so}(x;\sigma)=\frac{1}{2\sigma^2}\log^2(x+\sqrt{x^2-1})+\log(\log(x+\sqrt{x^2-1})).
\end{align}
Note that the first term is just $V_S$, the potential of the Siegel domain (up to rescaling of $\sigma$). For finite $N$, the second term changes things and the (skew-)orthogonal polynomials will differ from the ones for the Siegel domain. However, in the large $N$ limit, where $\sigma^2\rightarrow0$, this complicated second term gets suppressed so that the saddle-point equation will be of the same form as the saddle-point equation  \eqref{eq:saddle_S} of the Siegel domain.
\end{rem}

\section{Duality} \label{sec:duality}
Let $M \simeq G/H$ be a Riemannian symmetric space of non-compact type. Its dual symmetric space $M^\vee$ is a Riemannain symmetric space of compact type,
$M^\vee = G^*/H$. Here, if the Lie algebra $\mathfrak{g}$ of $G$ has Cartan decomposition $\mathfrak{g} = \mathfrak{h} + \mathfrak{p}$, then $G^*$ is the connected Lie group with Lie algera $\mathfrak{g}^* = \mathfrak{h} + i\mathfrak{p}$ (where $i = \sqrt{-1}$).

Recall the Riemannian symmetric space $M = \mathcal{P}_{\mathbb{C}}(N)$, which is given by $G =\GL(N,\mathbb{C})$ and $H = \U(N)$. In this case, $\mathfrak{p}$ is the space of Hermitian matrices (so $i\mathfrak{p} = \mathfrak{u}(N)$, the space of skew-Hermitian matrices) and $\mathfrak{g}^* = \mathfrak{u}(N) + \mathfrak{u}(N)$. Thus, one has the dual space $M^\vee = \U(N) \times \U(N)/\U(N) \simeq \U(N)$. 
In other words, the dual of the space $M = \mathcal{P}_{\mathbb{C}}(N)$ of positive-definite Hermitian matrices is the unitary group $M^\vee = \U(N)$.

Consider now a family of distributions on $M^\vee$, which will be called $\Theta$ distributions, and which display an interesting connection with Gaussian distributions on $M$, studied in \ref{subsec:exactz2}.
Recall Jacobi's $\vartheta$ function (to follow the original notation of Jacobi, this should be written as $\vartheta(e^{i\phi}|q)$ where $q = \exp(-\sigma^2)$, as in~\cite{watson}), 
$$
\vartheta(e^{\scriptscriptstyle i\phi}|\sigma^{\scriptscriptstyle 2}) \,=\, \sum^{+\infty}_{m=-\infty} \exp(-m^2\hspace{0.02cm}\sigma^2 + 2m\hspace{0.03cm}i\phi).
$$
As a function of $\phi$, up to some minor modifications, this is just a wrapped normal distribution (in other words, the heat kernel of the unit circle),
$$
\frac{1}{2\pi}\hspace{0.03cm}
\vartheta\!\left(e^{\scriptscriptstyle i\phi}|{\scriptstyle \frac{\sigma^2}{2}}\right) \,=\,  \hspace{0.03cm}\sum^{\infty}_{m=-\infty} \exp\left[ - \frac{(2\phi - 2m\pi)^2}{2\sigma^2}\right].
$$
Each $x \in M^\vee$ can be written as $x = k\cdot e^{i\theta}$ for some $k \in U(N)$ and $e^{i\theta} = \mathrm{diag}(e^{i\theta_i}\,;i=1,\ldots, N)$, where $k\cdot y = k\hspace{0.02cm}y\hspace{0.02cm}k^\dagger$, for $y \in M^\vee$ (  $^\dagger$ denotes the conjugate-transpose). With this notation, define the following matrix $\vartheta$ function,
\begin{equation} \label{eq:THETAF}
  \Theta\left(x\middle|\sigma^2\right) \,=\, k\cdot \vartheta\!\left(e^{\scriptscriptstyle i\theta}|{\scriptstyle \frac{\sigma^2}{2}}\right)
\end{equation}
which is obtained from $x$ by applying Jacobi's $\vartheta$ function to each eigenvalue of $x$. Further, consider the positive function,
\begin{equation} \label{eq:THETAD}
 f_*(x|\bar{x}\hspace{0.02cm},\sigma) \,=\, \det\left[\left( 2\pi\hspace{0.03cm}\sigma^2\right)^{\!\frac{1}{2}}\hspace{0.03cm}\Theta\!\left(x\bar{x}^\dagger\middle|\sigma^2\right)\right]
\end{equation}
which is also equal to 
$$
\det\left[\left( 2\pi\hspace{0.03cm}\sigma^2\right)^{\!\frac{1}{2}}\hspace{0.03cm}\Theta\!\left(\bar{x}^\dagger x\middle|\sigma^2\right)\right]
$$ 
since the matrices $x\bar{x}^\dagger$ and $\bar{x}^\dagger x$ are similar. Then, let $Z_{\scriptscriptstyle M^\vee}(\sigma)$ denote the normalising constant
\begin{equation} \label{eq:zstar}
 Z_{\scriptscriptstyle M^\vee}(\sigma) = \int_{M^\vee}f_*(x|\bar{x}\hspace{0.02cm},\sigma)\,\mathrm{vol}(dx)
\end{equation}
which does not depend on $\bar{x}$, as can be seen by introducing the new variable of integration $z = x\bar{x}^\dagger$ and using the invariance of $\mathrm{vol}(dx)$ (compare to Remark \ref{rem:z_invariance}). 

Now, define a $\Theta$ distribution $\Theta(\bar{x},\sigma)$ as the probability distribution on $M^\vee$, whose probability density function, with respect to $\mathrm{vol}(dx)$, is given by 
\begin{equation} \label{eq:thetadensity}
  p_*(x|\bar{x}\hspace{0.02cm},\sigma) \,=\,\left(Z_{\scriptscriptstyle M^\vee}(\sigma)\right)^{-1}\hspace{0.03cm}f_*(x|\bar{x}\hspace{0.02cm},\sigma).
\end{equation}
\begin{prop} \label{prop:thetadual}
  Let $Z_{\scriptscriptstyle M}(\sigma) = Z_2(\sigma)$, be given by Proposition \ref{thm:exactZ_2}, and $Z_{\scriptscriptstyle M^\vee}(\sigma)$ be given by (\ref{eq:zstar}). Then, the following equality holds
\begin{equation} \label{eq:thetadual}
  \frac{Z_{\scriptscriptstyle M}(\sigma)}{Z_{\scriptscriptstyle M^\vee}(\sigma)} = \exp\left[{\small\left(\frac{N^3 - N}{6}\right)}\sigma^{2}\right].
\end{equation}
\end{prop}
\begin{rem}
The Gaussian density (\ref{eq:gaussdensity}) on $M$, and the $\Theta$ density (\ref{eq:thetadensity}) on $M^\vee$ are apparently unrelated. Therefore, it is interesting to note that their normalising constants $Z_{\scriptscriptstyle M}(\sigma)$ and $Z_{\scriptscriptstyle M^\vee}(\sigma)$ scale together according to the simple relation (\ref{eq:thetadual}). The connection between the two distributions is due to the duality between the two spaces ($M$ and $M^\vee$). 
\end{rem}
\begin{proof}
Since $Z_{\scriptscriptstyle M^\vee}(\sigma)$ does not depend on $\bar{x}$, one may set $\bar{x} = o$ in (\ref{eq:zstar}), where $o$ stands for the $N \times N$ identity matrix. Then,  $f_*(x|o\hspace{0.02cm},\sigma)$ is a class function, depending only on the eigenvalues $(\theta_{\scriptscriptstyle 1\,},\ldots,\theta_{\scriptscriptstyle N})$ of $x$. Accordingly, (\ref{eq:zstar}) can be computed using the Weyl integral formula~\cite{meckes}
\begin{equation}\label{eq:invarmondeun}
\int_{M^\vee}f_*(x|o\hspace{0.02cm},\sigma)\,\mathrm{vol}(dx) \,=\,
\frac{\omega_{2}(N)}{2^{\scriptscriptstyle N^2}N!} \,\int_{[0\hspace{0.02cm},2\pi]^N}\,
f_*\left(\theta_{\scriptscriptstyle 1\,},\ldots,\theta_{\scriptscriptstyle N}|\bar{x}\hspace{0.02cm},\sigma\right)\hspace{0.03cm}|\Delta(e^{i\theta})|^2\hspace{0.03cm}d\theta_{\scriptscriptstyle 1}\ldots\theta_{\scriptscriptstyle N}
\end{equation}
where $\omega_2(N)$ is the same as in Proposition \ref{thm:exactZ_2}, and $\Delta(e^{i\theta})$ denotes the Vandermonde determinant (in the notation of (\ref{al:Z_beta})). Therefore,  
\begin{equation} \label{eq:proofduality1}
 Z_{\scriptscriptstyle M^\vee}(\sigma) = 
\frac{\omega_{\scriptscriptstyle 2}(N)}{\mathstrut 2^{\scriptscriptstyle N^2} N!}\left( 2\pi\hspace{0.03cm}\sigma^2\right)^{\!\frac{N}{2}}\times I_{\scriptscriptstyle 2}
\end{equation}
where $I_{\scriptscriptstyle 2}$ is the integral
\begin{equation} \label{eq:proofduality2}
I_{\scriptscriptstyle 2} \,=\,
\int_{[0\hspace{0.02cm},2\pi]^N}\,\prod^N_{i=1}\vartheta\!\left(e^{\scriptscriptstyle i\theta_i}|{\scriptstyle \frac{\sigma^2}{2}}\right)|\Delta(e^{i\theta})|^2\hspace{0.04cm}d\theta_{\scriptscriptstyle 1}\ldots\theta_{\scriptscriptstyle N}
\end{equation}
which follows from (\ref{eq:THETAD}) and the identity
$$
\det \Theta\!\left(x\middle|\sigma^2\right)
= \prod^N_{i=1}\vartheta\!\left(e^{\scriptscriptstyle i\theta_i}|{\scriptstyle \frac{\sigma^2}{2}}\right).
$$
Now, $I_{\scriptscriptstyle 2}$ can be expressed by using (\ref{eq:vmondorthpo}) to rewrite $\Delta(e^{i\theta})$, as in the proof of Proposition \ref{thm:exactZ_2}. Precisely, if $(p_{\hspace{0.02cm}n}\,; n = 0,1,\ldots)$ are orthonormal trigonometric polynomials, with respect to the weight function $\vartheta\!\left(e^{\scriptscriptstyle i\theta}|{\scriptstyle \sigma^2\!/2}\right)$, on the unit circle, then $I_{\scriptscriptstyle 2}$ is given by, 
$$
I_{\scriptscriptstyle 2} \,=\, N!\hspace{0.03cm}\prod^{N-1}_{n=0} p^{-2}_{\hspace{0.02cm}nn}
$$
in terms of the leading coefficients $p_{\hspace{0.02cm}nn}$ of the polynomials $p_{\hspace{0.02cm}n}$ (these leading coefficients may always be chosen to be real). At present, the required orthonormal polynomials $p_{\hspace{0.02cm}n}$ are given by
\begin{equation} \label{eq:rogersz1}
  p_{\hspace{0.02cm}n}(z) \,=\, \left[q^{n}\!\prod^n_{m=1}( 1 - q^{m})^{-1}\right]^{\!\frac{1}{2}}r_{\hspace{0.02cm}n}(-q^{-\frac{1}{2}}z)
\end{equation}
where $q = e^{-\sigma^2}$ and $r_n(z)$ is the $n$-th Rogers-Szegö polynomial, which is monic~\cite{rogers}. Therefore, 
\begin{equation} \label{eq:rogerspnn}
   p^{-2}_{\hspace{0.02cm}nn} \,=\, \prod^{n}_{m=1}\left(1 - e^{-m\hspace{0.02cm}\sigma^2} \right)
\end{equation}
and, from (\ref{eq:proofduality2}), $I_{\scriptscriptstyle 2}$ is given by
\begin{equation} \label{eq:szegoi2}
I_{\scriptscriptstyle 2} \,=\, N! \prod^{N-1}_{n=1}\left(1 - e^{-n\hspace{0.02cm}\sigma^2} \right)^{\!N-n}
\end{equation}
which may be replaced into (\ref{eq:proofduality1}) to obtain
\begin{equation} \label{eq:zstarformula}
 Z_{\scriptscriptstyle M^\vee}(\sigma) = 
\frac{\omega_{\scriptscriptstyle 2}(N)}{\mathstrut 2^{\scriptscriptstyle N^2}}\left( 2\pi\hspace{0.03cm}\sigma^2\right)^{\!\frac{N}{2}}\hspace{0.02cm} \prod^{N-1}_{n=1}\left(1 - e^{-n\hspace{0.02cm}\sigma^2} \right)^{\!N-n}.
\end{equation}
Finally, (\ref{eq:thetadual}) follows easily, by comparing (\ref{eq:zstarformula}) to Proposition \ref{thm:exactZ_2}.
\end{proof}

\section{Conclusion}
\label{sec:Outlook}

We have seen that there are many interesting structures hidden behind the seemingly intractable normalisation factors of Gaussian distributions on symmetric spaces. We have focused primarily on the exact computation of $Z_1$ and $Z_S$ for finite $N$ as well as their approximations for large $N$ (the so-called large $N$ limit). Our calculations in these cases were as explicit as possible in order to directly apply them in future projects. In particular, we hope to apply our results to real brain data with $N\in \{50,1000\}$ in future BCI applications. One direction concerns numerical experiments with the formulas given in this paper and their incorporation into the algorithms of \cite{Sa17}\cite{Sa16}. The exact $\beta=2$ formula from Section \ref{sec:finite_N} provides a testing ground to numerically estimate a value $N^*$ such that the large $N$ limit becomes a good approximation for $N>N^*$. Moreover, we can try to compute enough skew-orthogonal polynomials numerically to use the exact formulas from Section \ref{sec:finite_N} in order to compute $Z_1$ for $N<N^*$. For $N>N^*$ we can then approximate $Z_1$ by a large $N$ limit using the formulas from Section \ref{sec:large_N}. This would effectively enable us to compute $Z_1$ for all $N$. A similar paradigm might lead to computations of $Z_S$ for arbitrary $N$. We hope that this work will open the path for future research on the development and application of probabilistic modelling and statistical analysis on symmetric spaces appearing in a variety of engineering applications involving high-dimensional data.


\appendices

\section{Taking care of the prefactors}
\label{subsec:prefactors}

\begin{lem}
\label{lem:Ccalc}
Let $\beta \in \{1,2,4\}$. Define \begin{align}
C_{N,\beta}(\sigma):=\frac{\omega_{\beta}(N)(2\pi)^N}{2^{NN_\beta}}\exp(-NN^2_{\beta} \frac{\sigma^2}{2})
\end{align} 
and \begin{align}
\label{eq:C_infty}
    C_{\infty, \beta}(t):=\exp\Big(N^2\big(-\frac{\beta^2}{8}t+\frac{1}{N}\big(\log(\pi)+\beta(-\frac{\log(2)}{2}+\frac{\beta t}{4}- \frac{t}{2})\big)+\frac{1}{N^2}\log(\omega_{\beta}(N))\big)\Big).
\end{align}
Then, in the large $N$ limit, we get \begin{align}
    \frac{1}{N^2}\log \big( C_{N,\beta}(\sigma) \big) \sim \frac{1}{N^2} \log\big( C_{\infty, \beta}(\sigma) \big)+ \mathcal{O}(\frac{1}{N^2}).
\end{align}
\end{lem}

\begin{proof}
The lemma follows from the following simple calculation:
\begin{align}
    \frac{1}{N^2}\log(C_{N,\beta}(\sigma))&= \frac{1}{N^2}\log(\omega_{\beta}(N))+\frac{\log(\pi)}{N}-\frac{N_{\beta}}{N^2}\log(2)-\frac{N^2_{\beta}\sigma^2}{2N}\\
    &=-\frac{\beta^2}{8}t+\frac{1}{N}\big(\log(\pi)+\beta(-\frac{\log(2)}{2}+\frac{\beta t}{4}- \frac{t}{2})\big) +\frac{1}{N^2}\log(\omega_{\beta}(N))+ \mathcal{O}(N^{-2}).
\end{align}
\end{proof}

In a certain sense, we do not have to worry about the $\sigma$-independent terms in equation \eqref{eq:C_infty} since for the applications of Section \ref{subsec:Moti}, we are only interested in derivatives of the free energy. Anyway, it will be convenient to include formulas for $\omega_{\beta}(N)$ for $\beta\in \{1,2\}$ to explicitly relate $Z_1$ and $Z_2$. Similar formulas exist for $\omega_4(N)$ but we will not discuss them for the reason mentioned above.

\begin{lem}
\label{lem:omegas}
$\omega_{1}(N)$ and $\omega_2(N)$ are given by 
\begin{align}
\label{eq:omega_fin}
    \omega_{1}(N)&=\frac{\vol(\textrm{O}(N))}{(2\pi)^N}=\pi^{\frac{N(N-3)}{4}}\prod \limits_{K=1}^N\frac{1}{\Gamma(\frac{k}{2})}=\frac{\pi^{\frac{N(N-2)}{2}}}{\Gamma_N(\frac{N}{2})}=\pi^{\frac{N(N-3)}{4}} \frac{G(\frac{1}{2})G(1)}{G(\frac{n+1}{2})G(\frac{n+2}{2})}\\
    \omega_2(N)&=\frac{\vol(\U(N))}{(2\pi)^N}=\frac{(2\pi)^{\frac{N(N+1)}{2}}}{G(N)}
\end{align}
where the \textbf{Barnes-function} $G$ is defined by $G(N)=\prod \limits_{k=1}^{N-1}k!$ for $N\in \mathbb{N}$ and can be extended to all of $\mathbb{C}$. $\Gamma_N$ is the so-called \textbf{multivariate gamma function}.

\end{lem}
\begin{proof}
The formula for $\omega_1(N)$ can be found in Proposition 2.24 of \cite{Zh15}.
The intuition behind these formulas comes from realising that the column vectors $A_i\coloneqq(A_{ij})_{j\in\{1,\dots,N \}}$ of $A\in \textrm{O}(N)$ form an orthonormal basis. This means that $A_1\in S^N$, $A_2 \in S^N \cap \{A_1\}^{\perp}\cong S^{N-1}$ and so on. Hence, $\vol(\textrm{O}(N))= \sum \limits_{k=1}^N\vol( S^k)$, which together with $\vol(S^{k-1})=2\pi^{\frac{k}{2}}/\Gamma(\frac{k}{2})$ gives the result. Analogously, $\vol(\U(N))= \sum \limits_{k=1}^N \vol(S^{2k-1})$. Also, $\vol(\Sp(N))\cong  \prod \limits_{k=1}^N \vol(S^{4k-1})$, which yields the equations in \eqref{eq:omega_fin}.
\end{proof}

Combining Lemma \ref{lem:omegas} with the following Stirling-like expansion of the Barnes function
\begin{align}
    \frac{1}{z^2}\log(G(z+1))= \frac{\log(z)}{2}-\frac{3}{4}+ \frac{1}{2z}\log(2\pi)- \frac{1}{12z^2} \log(z) + \mathcal{O}(z^{-2})
\end{align}
gives us explicit formulas for the large $N$ limit of $\omega_{\beta}(N)$ such as
\begin{align}
 \frac{1}{N^2}\log( \omega_2(N)) &\sim - \frac{1}{2} \log( \frac{N}{2\pi}) + \frac{3}{4}+ \frac{1}{12} \log(N) + \mathcal{O}(N^{-2}).
\end{align}
A similar formula holds for $\beta=1$.

\section*{Acknowledgment}
S.H. is supported by the Science and Technology Facilities Council (STFC) and St. John's College, Cambridge. S.H. also benefited from partial support from the Cambridge Mathematics Placement (CMP) Programme and the European Research Council under the Advanced ERC Grant Agreement Switchlet n.670645. C.M. is supported by Fitzwilliam College and a Henslow Fellowship from the Cambridge Philosophical Society. The authors would also like to thank Leong Khim Wong for producing the contents of Figure \ref{fig:feynman} and allowing us to use it in our paper.

\ifCLASSOPTIONcaptionsoff
  \newpage
\fi



\bibliographystyle{IEEEtran}
\bibliography{References}

\end{document}